\documentclass[11pt]{article}
\usepackage{amsfonts}
\usepackage{amsthm}
\usepackage{amsmath}
\usepackage{amssymb}
\usepackage{authblk}
\usepackage{xcolor}
\usepackage{appendix}
\usepackage{url}
\usepackage[margin=0.9in]{geometry}
\usepackage{titlesec}
\usepackage{booktabs}
\usepackage[T1]{fontenc}
\usepackage[utf8]{inputenc} % or remove entirely if using lualatex/xelatex
\usepackage[mathscr]{eucal}
\usepackage{indentfirst}
\usepackage{graphicx}
\usepackage{graphics}
\numberwithin{equation}{section}
\usepackage{epstopdf} 
\usepackage{xcolor}
\usepackage[colorlinks=true,
            linkcolor=blue,
            urlcolor=blue,
            citecolor=blue]{hyperref}
\usepackage{algorithm}
\usepackage{algpseudocode}  % Part of algorithmicx
\usepackage[
    backend=biber,
    style=nature,
    maxcitenames=3,
    maxbibnames=3,
]{biblatex}
\theoremstyle{plain}
\newtheorem{Th}{Theorem}[section]
\newtheorem{Lemma}[Th]{Lemma}
\newtheorem{Cor}[Th]{Corollary}
\newtheorem{Prop}[Th]{Proposition}

\theoremstyle{definition}
\newtheorem{Def}[Th]{Definition}

\newtheorem{Res}[Th]{Result}

\newcommand{\red}{\color{black}} % toggle if necessary
\newcommand{\bsalpha}{\boldsymbol{\alpha}}
\newcommand{\bsbeta}{\boldsymbol{\beta}}
\newcommand{\bsphi}{\boldsymbol{\phi}}

\newcommand{\bsvartheta}{\boldsymbol{\vartheta}}
\newcommand{\bsa}{\boldsymbol{a}}
\newcommand{\bsb}{\boldsymbol{b}}
\newcommand{\bsu}{\boldsymbol{u}}
\newcommand{\bsv}{\boldsymbol{v}}

\newcommand{\Var}{\text{Var}}
\newcommand{\software}[1]{{\color{brown!65!black}\textsc{#1}}}

\title{Asymptotic enumeration of admixed arrays and a different independence heuristic}
\author[1]{Alan J.~Aw\thanks{Email: \href{mailto:alan.aw@pennmedicine.upenn.edu}{alan.aw@pennmedicine.upenn.edu}}}
\affil[1]{Department of Genetics, Perelman School of Medicine, University of Pennsylvania}
\date{\today}

\begin{document}
\maketitle

\begin{abstract} 
We introduce a class of paired binary matrices called admixed arrays, which arise in {\red analyses} of large-scale genetic data and can be viewed as weighted edge colorings of complete bipartite graphs. This combinatorial structure gives rise to two natural families of marginal constraints: a row-sum constraint and a paired column-sum constraint, the latter inducing an inequality among entries of the matrix pair. We study the enumeration of admixed arrays under these constraints in dense regimes. First, we obtain exact formulas for the sizes of the families defined by each constraint in isolation and derive a finite-size criterion characterizing when one constraint is more restrictive than the other. In the large-dimension limit, this comparison simplifies to an entropy inequality, yielding an information-theoretic interpretation and a quantifiable error bound in the semi-regular case. We then analyze the asymptotic enumeration of the doubly constrained family in a semi-regular setting. Using saddle-point approximation and probabilistic techniques, we derive a detailed asymptotic expansion for the logarithm of the count, isolating an explicit fourth-moment contribution and establishing quantitative control of the higher-order remainder. A consequence of this analysis is a phenomenon absent from classical binary and integer matrix models: in the regime $N=\Theta(P)$ with uniform margins and density bounded away from zero, the two constraint families obey the independence heuristic with a correction factor $1/\sqrt[4]{e}$ rather than the familiar $e^{\pm1/2}$. Numerical experiments corroborate the analytical approximations, and we implement and extend an algorithm of Miller and Harrison (2013) as open-source software to enumerate constrained admixed arrays.
\end{abstract}

%\noindent\textbf{Keywords:} binary matrix; entropy; asymptotic enumeration; saddle-point approximation 

%\noindent\textbf{Mathematics Subject Classification (2010):} 05A16, 

\section{Introduction}

For positive integers $N$ and $P$, let $\mathscr{A}(\mathbf{r},\mathbf{c})$ denote the set of $N\times P$ binary matrices with some prescribed row and column sums, $\mathbf{r}\in\mathbb{N}^N$ and $\mathbf{c}\in\mathbb{N}^P$. Enumerating $\mathscr{A}$ --- equivalently, counting bipartite graphs with a given degree sequence --- is a central topic in combinatorics and probability. In dense regimes, where the prescribed constraints grow proportionally in the dimensions, asymptotic enumeration is guided by maximum-entropy principles or saddle-point methods. A well-known outcome of this line of work is the \emph{independence heuristic}, which states that distinct families of marginal constraints interact asymptotically independently up to an essentially dimension-free correction factor, provided the constraints are \emph{semi-regular}, i.e., uniform along each margin. The heuristic also applies to integer matrices with prescribed margins, and remains valid in sparse regimes and weak perturbations under the dense regime; we summarize representative results below.        

\begin{Prop}[Independence heuristic for binary and integer matrices]\label{prop:indep-heur}
Let $\mathbf{r}=(r_1,\ldots,r_N)$ and $\mathbf{c}=(c_1,\ldots,c_P)$ be the prescribed row and column sums of $\mathbf{A}\in\{0,1\}^{N\times P}$. Let $\lambda=(c_1+\ldots+c_P)/(NP)=(r_1+\ldots+r_N)/(NP)$ be the fraction of entries in the matrix which are $1$, and define the normalization factor $D={NP\choose \lambda NP}$. If $\mathscr{A}_\mathbf{r}$ and $\mathscr{A}_\mathbf{c}$ denote the sets of binary matrices with prescribed row sums $\mathbf{r}$ and prescribed column sums $\mathbf{c}$, then for large $N$ and $P$ with $N=\Theta(P)$ and $\lambda$ bounded away from $0$,
\begin{equation}
|\mathscr{A}|= k_{N,P}D^{-1} |\mathscr{A}_\mathbf{r}||\mathscr{A}_\mathbf{c}|,\label{eq:indep-heur}
\end{equation}
with correction factor $k_{N,P}>0$ satisfying $\lim_{N,P\to\infty} k_{N,P}=1/\sqrt{e}$, provided that (1) $r_1=\cdots=r_N$ and $c_1=\cdots=c_P$ \cite{canfield2005asymptotic}; or (2) $\max_{n\in[N]} r_n \ll P$ and $\max_{p\in[P]} c_p \ll N$ \cite{greenhill2006asymptotic}; or (3) $|r_n-\lambda P|=O(P^{1/2+\varepsilon})$ and $|c_p-\lambda N|=O(N^{1/2+\varepsilon})$ for some $\varepsilon >0$ \cite{canfield2008asymptotic,liebenau2023asymptotic}. For integer matrices with prescribed row and column sums, an analogue of Eq.~\eqref{eq:indep-heur} holds with $\lim_{N,P\to\infty} k_{N,P}=\sqrt{e}$, under similar restrictions to the row and column sum entries \cite{greenhill2008asymptotic,canfield2010asymptotic}.  
\end{Prop}
Other correction factors are known for incidence matrix structures (e.g., regular graphs, uniform hypergraphs and directed graphs), which lie outside the bipartite, rectangular setting considered here (see, e.g., \cite{wormald2018asymptotic}). Given that the classical bipartite and rectangular matrix models all exhibit a $e^{\pm 1/2}$ correction factor whenever $N=\Theta(P)$ and $\lambda=\Theta(1)$, it is natural to treat $e^{\pm 1/2}$ as a benchmark for independence-based asymptotics in dense, semi-regular settings. Here, we show that this benchmark intuition requires refinement. We introduce a class of discrete structures defined by pairs of binary matrices, for which the independence heuristic holds under an analogously dense and semi-regular regime, but with a \emph{different correction factor} of $1/\sqrt[4]{e}$. The departure from the $e^{\pm 1/2}$ benchmark reflects the way the column constraint couples the two matrices, which induces feasibility inequalities among entries of the pair and alters both exact enumeration and the resulting asymptotic correction.

Concretely, our discrete structures are admixed arrays, defined as pairs of $N\times 2P$ binary matrices $[\mathbf{A},\mathbf{X}]$. Let $\mathscr{A}_0(N,P)$ denote the set of all admixed arrays. We introduce two constrained families, $\mathscr{A}_1\subset \mathscr{A}_0$ and $\mathscr{A}_2\subset \mathscr{A}_0$, which are defined by a row sum constraint (\emph{global ancestry}) and a paired column sum constraint (\emph{ancestry-specific allele dosage}). While these constraints are not standard in classical combinatorial models, they arise naturally in analyses of large-scale genetic data. For this reason we postpone their exact definition to Section \ref{sec:preliminaries}, where we explicitly provide statistical- and population-genetic interpretation. We also consider the doubly constrained family $\mathscr{A}_{12}=\mathscr{A}_1\cap \mathscr{A}_2$. 

As a preparatory step, we derive explicit formulas for $|\mathscr{A}_1|$ and $|\mathscr{A}_2|$, yielding a finite-size criterion for comparing the relative restrictiveness of the row-sum and paired column-sum constraints. In the dense regime, this comparison admits a clean entropy approximation with a quantifiable error rate in the semi-regular case. 

\begin{Res}[Informal statement of Theorems \ref{thm:asymp_compare_a1_a2} and \ref{thm:convergence_uniform}]\label{res:1}
Let $H_1$ denote the mean global ancestry entropy and $H_2$ denote the mean ancestry-specific allele fraction entropy; they are functions of the row sum and paired column sum constraints. Let $\overline{f}$ denote the sum of mean ancestry-specific allele fractions, which is a function of the paired column sum constraint. An approximate criterion for $|\mathscr{A}_1| > |\mathscr{A}_2|$ is $H_1 > H_2-\overline{f}$. Moreover, in case the row and paired column sums are uniform, the classification error fraction of the approximate criterion goes to $0$ at a $\sqrt{\log_2 (N)/N + \log_2(P) /P}$ rate.     
\end{Res}

These single-constraint results isolate the combinatorial and entropic effects of each restriction, providing intuition for how the paired column constraint departs from standard margin constraints. With this intuition in hand, we next turn to the asymptotic enumeration of the doubly constrained family $\mathscr{A}_{12}$. %A key distinction of the doubly constrained family is that the paired column constraint couples $\mathbf{A}$ and $\mathbf{X}$, restricting admissible column patterns and thus modifying the effective constraint structure relative to classical margin-constrained matrices. 
Focusing on a semi-regular setting with scaled uniform margins equal to $1/2$, we obtain a detailed expansion using saddle-point approximation and probabilistic techniques. 

\begin{Res}[Informal statement of Theorem \ref{thm:growth-rate-a12}]\label{res:2}
Let the row sum and paired column sum constraints be uniform and equal to $P$ and $N$ respectively. For large $N$ and $P$ with $N=\Theta(P)$, 
\begin{align*}
\log_2|\mathscr{A}_{12}|
&= 2NP
- \frac{1}{2}\left[
N\log_2(\pi P)
+ P\log_2(\pi N)
- \log_2(\pi N P)
\right] \\
&\quad
- \frac{(N+P-1)^2}{8NP}\log_2(e)
+ O\!\left(\frac{1}{\sqrt{\min\{N,P\}}}\right).
\end{align*}
\end{Res}

Our proofs are analytic and probabilistic. Result \ref{res:1} relies on an $\varepsilon$-regularization argument combined with a geometric analysis exploiting the concavity of the entropy function, which may be of independent interest for obtaining error fractions in entropy-based approximate comparisons of constrained combinatorial families. Result \ref{res:2} adapts saddle-point approximation techniques from earlier works to the admixed array setting. Even in the semi-regular $1/2$ case, where the inequality constraint collapses to an equality, the resulting problem still involves extracting a coefficient from a Laurent polynomial and handling a rank-deficient Cauchy integral. We construct invariant transformations of measure to reduce the effective dimension, and apply concentration inequalities and hypercontractivity to obtain the final approximation. In this regime, we show that the independence heuristic $D^{-1}|\mathscr{A}_1||\mathscr{A}_2|$ is asymptotically $\sqrt[4]{e}\cdot|\mathscr{A}_{12}|$, implying a $1/\sqrt[4]{e}$ correction. Finally, we implement and extend an algorithm of \textcite{miller2013exact} to evaluate our approximate criterion and saddle-point approximation, observing good numerical agreement. Our algorithms are available as open-source software.       

\subsection{Previous Work on Constrained Matrix Enumeration}
\label{subsec:lit_review}
The enumeration of row- and column-sum constrained sets of binary matrices, $\mathscr{A}(\mathbf{r},\mathbf{c})$, is a classical problem that dates back to the later half of the 20th century (see, e.g., \cite{litsyn2003capacity,ordentlich2000two,mckay1984asymptotics,mckay1990asymptotic,wormald1980some,good1977enumeration,mineev1976on,bender1974asymptotic,bekessy1972asymptotic,everett1971asymptotic,o1969asymptotics,read1959some}). The Gale-Ryser criterion \cite[Theorem 16.1]{van2001course} determines if such matrices exist (i.e., if $|\mathscr{A}|>0$ given the constraints), but computing $|\mathscr{A}|$ for non-empty sets is more challenging. The independence heuristic was first noted by \textcite{good1977enumeration} and, as described in Proposition \ref{prop:indep-heur}, holds for a wide range of row and column constraints. While the conditions stated in Proposition \ref{prop:indep-heur} imply $\lim_{N,P\to\infty}k_{N,P}\to e^{\pm1/2}$, other uniform regimes lead to different correction factors. These include assuming fixed finite column and row sums while sending the matrix dimensions to infinity \cite{bekessy1972asymptotic}, or more generally sending the dimensions $N$ and $P$ to infinity at different rates while imposing a small enough density \cite{greenhill2006asymptotic}. The independence heuristic also fails in certain cases: for dense binary matrices with carefully constructed \emph{non-uniform} constraints \cite{barvinok2010number,wu2020asymptotic}, the correction factor can be exponentially decaying (e.g., $k_{N,P}\sim 1/2^{\Omega(NP)}$ in Proposition \ref{prop:indep-heur}). Analogously, for dense integer matrices the correction factor can be exponentially large \cite{barvinok2009asymptotic,lyu2022number}. These results also feature a variety of proof techniques, including degree switching and contraction mapping \cite{greenhill2006asymptotic,liebenau2023asymptotic}, Lorentzian polynomials \cite{gurvits2015boolean,branden2023lower}, maximum entropy and typical tables \cite{barvinok2010number,wu2020asymptotic,lyu2022number}, saddle-point approximation \cite{canfield2008asymptotic,canfield2005asymptotic} and its closely related complex martingale theory \cite{isaev2018complex}. Apart from theoretical studies, other works have investigated efficient algorithms for exact or approximate enumeration (\cite{branden2023lower,gurvits2015boolean,miller2013exact,barvinok2010number} and citations therein).

\subsection{Organization of the Paper}
Section \ref{sec:preliminaries} introduces general admixed arrays, defines the constraint families relevant to our present study, and states technical lemmas required to prove our results. Section \ref{sec:single-constraint} analyzes the effects of the row-sum and paired column-sum constraints in isolation and establishes Result \ref{res:1}, setting the stage for doubly constrained enumeration. In Section \ref{sec:results} we prove Result \ref{res:2} and derive the associated independence heuristic. We also describe algorithms we implemented for exact computation, and perform numerical comparisons of our approximations. We conclude with a discussion of future directions.             
\begin{figure}[ht]
\begin{center}
\includegraphics[width=0.95\textwidth]{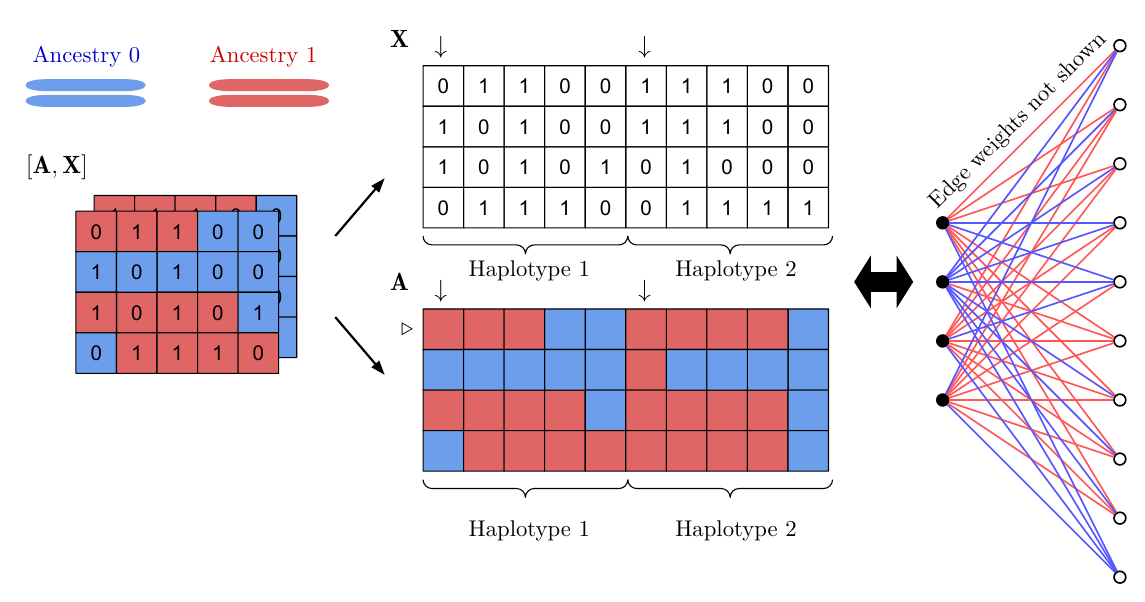}
\end{center}
\caption{An example of a two-way ($\ell=2$) admixed array for diploid individuals ($k=2$) with $N=4$ individuals (rows) and $P=5$ loci (paired columns). The homologous haplotypes, which determine the pairing of columns, are annotated for clarity. Throughout this work, we assign one color to ``Ancestry 0'' and another to ``Ancestry 1'' (in this case, $0=$~blue and $1=$~red) so that global ancestries are determined with respect to Ancestry 1. In this example, the first individual's global ancestry, $\overline{A_{1\cdot}}$, is computed using only information in the local ancestry matrix $\mathbf{A}$ --- specifically, by counting the fraction of red cells along the first row (marked by $\triangleright$), $\overline{A_{1\cdot}}=7/10=0.7$. On the other hand, ancestry-specific allele frequencies for the first locus, $F_{1,0}$ and $F_{1,1}$, are computed using information in both $\mathbf{A}$ and the allele dosage matrix $\mathbf{X}$. Specifically, focusing on just the two columns corresponding to the first locus (these are marked by $\downarrow$), $F_{1,0}=1/2=0.5$ and $F_{1,1}=3/6=0.5$. The corresponding ancestry-specific allele dosages are $\Phi_{1,0}=1$ and $\Phi_{1,1}=3$. On the right, we further show the correspondence of $[\mathbf{A},\mathbf{X}]$ with a weighted edge coloring of a complete bipartite graph, where the black and white nodes correspond to the rows and columns of the admixed array. We omit edge weight labels to improve visualization.}
\label{fig:example}
\end{figure}     

\section{Model and Preliminaries}
\label{sec:preliminaries}

\subsection{Notation}
\label{subsec:math_concepts}

Vectors and matrices are always boldfaced (e.g., $\mathbf{A}$ denotes a matrix and $\mathbf{v}$ is a vector). For any two vectors $\mathbf{v}=(v_1,\ldots,v_P)$ and $\mathbf{w}=(w_1,\ldots,w_P)$ of the same length $P$, we write $\mathbf{v}\geq \mathbf{w}$ when each entry of the left vector is at least as large as the corresponding entry of the right vector (that is, $v_p\geq w_p$ for $p=1,\ldots,P$). Apart from standard notation ($\mathbb{R},\mathbb{N}$ and $[N]=\{1,\ldots,N\}$, etc.), sets are denoted by calligraphic script, as in $\mathscr{A}$, with their sizes denoted either by the insertion of flanking vertical bars (e.g., $|\mathscr{A}|$) if dealing with counting measure or interval lengths, or by using $\text{Vol}(\cdot)$ or $\mu(\cdot)$ otherwise (e.g., $\text{Vol}(\mathscr{B})$ or $\mu(\mathscr{R})$). $\log(\cdot)$ always denotes the base-$2$ logarithm and $\ln(\cdot)$ denotes the natural logarithm. For asymptotics notation, we follow \textcite{graham94}.  

\subsection{Admixed Arrays and Constraints Motivated by Genetic Data Analysis}
\label{subsec:biol_motivation}

In statistical and population genetics, admixed populations are groups of individuals that arise from genetic admixture \cite{korunes2021human,tan2023strategies}. This results in a mosaic of segments within the genome of a recently admixed individual, whereby each locus carries not only allele dosage information but also local ancestry information. Genetic data sequenced from admixed populations are summarized by an \emph{allele dosage matrix} $\mathbf{X}$ and a \emph{local ancestry matrix} $\mathbf{A}$. Both $\mathbf{A}$ and $\mathbf{X}$ have dimension $N\times kP$, where $N$ denotes the number of individuals (rows) in the sample, $P$ denotes the number of loci (disjoint groupings of columns), and $k$ is the ploidy number --- or number of homologous haplotypes per individual --- of the organism. In effect, the $kP$ columns are grouped into $P$ loci, where the $p$th locus corresponds to the column indices $\{p,P+p,\ldots ,(k-1)P+p\}$. Each entry of $\mathbf{X}$ is zero or one, but the number of possible values of each entry of $\mathbf{A}$ is determined by the number of source populations, $\ell\in\mathbb{N}$.  
\begin{Def}[$\ell$-way admixed array]\label{dfn:admix-array}
    For $k,\ell\geq 2$, let $\mathbf{A}=(A_{np})\in\{0,1,\ldots,\ell-1\}^{N\times kP}$ and $\mathbf{X}=(X_{np})\in\{0,1\}^{N\times kP}$ be $N\times kP$ integer matrices. The pair $[\mathbf{A},\mathbf{X}]$ is an $\ell$-way admixed array. 
\end{Def}

This work focuses on $(k,\ell)=(2,2)$, or \emph{two-way admixed arrays for diploid populations}, which will be referred to as admixed arrays for the rest of the paper. An example of an admixed array with $N=4$ individuals (rows) and $P=5$ loci (disjoint pairs of columns) is shown in Figure \ref{fig:example}. From Definition \ref{dfn:admix-array}, we may view an admixed array as a tensor, by stacking $\mathbf{A}$ and $\mathbf{X}$ along a third axis. A more interesting viewpoint is obtained by treating $\mathbf{A}$ as the biadjacency matrix of a bipartite graph $G=(\mathscr{U}\cup\mathscr{V},\mathscr{E})$ with $|\mathscr{U}|=N$ and $|\mathscr{V}|=2P$. Then, $[\mathbf{A},\mathbf{X}]$ can be viewed as a weighted edge coloring of the complete bipartite graph $K_{N,2P}$, where the edges are colored according to their corresponding value in $\mathbf{A}$ and assigned edge weights based on the corresponding value in $\mathbf{X}$. This viewpoint illustrates the close relationship between admixed arrays and classical bipartite graph models. 

Denote by $\mathscr{A}_0(N,P)$ the collection of all admixed arrays parameterized by the dimensions $N$ and $P$. We are interested in the following three scenarios, all of which place constraints on $\mathscr{A}_0$. Figure \ref{fig:example} (caption) calculates the constrained quantities for a simple example. 
\begin{enumerate}
\item \emph{Row local ancestry tally constraint.} For each individual $n$, the quantity $\overline{A_{n\cdot}}=\frac{1}{2P}\sum_{p=1}^{2P}A_{np}$ measures their global ancestry. The distribution of $\overline{A_{n\cdot}}$ is used to visualize contributions of ancestral populations to each individual \cite{alexander2009fast,pritchard2000inference} and to control for confounding in genetic mapping \cite{sun2025opportunities}, while summary statistics of $\{\overline{A_{n\cdot}}:n=1,\ldots,N\}$ are computed in studies of polygenic prediction \cite{aw2025hidden}. Here, we study its scaled integer-valued counterpart, $A_{n\cdot}=2P\times \overline{A_{n\cdot}}\in\mathbb{N}$, which we call the \emph{row local ancestry tally}. For each $n$, $A_{n\cdot}$ is the $n$th row sum of $\mathbf{A}$. We are interested in the size of $\mathscr{A}_1=\mathscr{A}_1(N,P;(a_{1\cdot},\ldots,a_{N\cdot}))=\left\{[\mathbf{A},\mathbf{X}]\in \mathscr{A}_0:~(\forall n\in [N])(A_{n\cdot}=a_{n\cdot})\right\}$.        
\item \emph{Ancestry-specific allele dosage (paired column sum) constraints.} For each locus $p\in\{1,\ldots,P\}$, its ancestry-specific allele frequencies are given by the two quantities, 
\begin{eqnarray}
F_{p,0}&=&\frac{\sum_{n=1}^N \left[(1-A_{np})X_{np}+(1-A_{n(P+p)})X_{n(P+p)}\right]}{\sum_{n=1}^N(2-A_{np}-A_{n(P+p)})},\label{eq:Fp0} \\
F_{p,1}&=&\frac{\sum_{n=1}^N \left(A_{np}X_{np}+A_{n(P+p)}X_{n(P+p)}\right)}{\sum_{n=1}^N (A_{np}+A_{n(P+p)})}.\label{eq:Fp1}
\end{eqnarray}
Eqs.~\eqref{eq:Fp0} and \eqref{eq:Fp1} arise from the use of both parental haplotypes to estimate allele frequencies in practice, and they are useful for interpreting variant effects in clinical genetics \cite{kore2025improved,barberena-jonas2026}. Here, we study their scaled integer-valued counterparts $\Phi_{p,0}=F_{p,0}\sum_{n=1}^N(2-A_{np}-A_{n(P+p)})$ and $\Phi_{p,1}=F_{p,1}\sum_{n=1}^N(A_{np}+A_{n(P+p)})$, which we call the \emph{ancestry-specific allele dosages}. Let $\boldsymbol{\phi}_0=(\phi_{1,0},\ldots,\phi_{P,0})$ and $\boldsymbol{\phi}_1=(\phi_{1,1},\ldots,\phi_{P,1})$ be two integer-valued vectors. We are interested in the size of $\mathscr{A}_2=\mathscr{A}_2(N,P;\boldsymbol{\phi}_0,\boldsymbol{\phi}_1)=\left\{[\mathbf{A},\mathbf{X}]\in\mathscr{A}_0:~(\forall p\in[P])(\Phi_{p,0}=\phi_{p,0} \wedge \Phi_{p,1}=\phi_{p,1})\right\}$.  
\item \emph{Both row local ancestry tally and ancestry-specific allele dosage (row and paired column sum) constraints.} We are also interested in the set of all admixed arrays constrained by the two quantities described earlier, i.e., $\mathscr{A}_{12}=\mathscr{A}_{12}(N,P; \bsphi_0,\bsphi_1,(a_{1\cdot},\ldots,a_{N\cdot}))=\mathscr{A}_1\cap \mathscr{A}_2$.   
\end{enumerate} 

\subsection{Auxiliary Lemmas}

We rely on some self-contained lemmas to prove our results. Their proofs are provided in Appendix \ref{appsec:aux-lemmas} for completeness. They can be skipped on first reading, and we will refer to them whenever they are used in a proof. 

\begin{Lemma}[Properties of the $\theta$ function]
\label{lemma:theta-function}
The function
\begin{equation*}
\theta(f_0,f_1)= f_0\log\frac{1}{f_0} + f_1\log\frac{1}{f_1} + (1-f_0-f_1)\log\left(\frac{1}{1-f_0-f_1}\right) - (f_0+f_1)
\end{equation*}
is strictly concave on the set $\mathscr{D}=\{(f_0,f_1)\in[0,1]^2:f_0+f_1\leq 1\}$, with maximum value $1$ attained at $(1/4,1/4)$.
\end{Lemma}

\begin{Lemma}[Ratio approximation]
\label{lemma:ratio-approx}
Let $\varepsilon>0$ be small. For any positive constant $c$, the quantity $\varepsilon\big/\ln\left[(1/2+\sqrt{c\varepsilon})/(1/2-\sqrt{c\varepsilon})\right]$ is roughly $\frac{1}{2}\sqrt{\varepsilon/c}$.
\end{Lemma}

\begin{Lemma}[Quadratic upper bound]
\label{lemma:quadratic-upper-bound}
Let $t\in[-\pi,\pi]$ and $h(t)=4\cos^2(t/2)$. Then $\ln\left(h(t)/4\right) \leq -t^2/4$.
% \begin{equation*}
% \ln\left(\frac{h(t)}{4}\right) \leq -\frac{t^2}{4}.
% \end{equation*}
\end{Lemma}

\begin{Lemma}[Properties of a square matrix]
\label{lemma:covariance-matrix-properties}
For positive integers $N$ and $P$, define the $(N+P-1)\times (N+P-1)$ matrix $\mathbf{B}$ as in Eq.~\eqref{eq:B-matrix}. The following are properties of $\mathbf{B}$.
\begin{enumerate}
\item It is invertible and has determinant $\det(\mathbf{B})=N^{P-1}P^{N-1}$.
\item Let $\mathbf{\Sigma}=2\mathbf{B}^{-1}$. The diagonal entries of $\mathbf{\Sigma}$ are $\Sigma_{n,n} = 2(N+P-1)/NP$ for $n\in[N]$ and $\Sigma_{(N+p),(N+p)}=4/N$ for $p\leq P-1$.
\item Let $\mathbf{\Sigma}$ be defined as in 2. Then the off-diagonal entries of $\mathbf{\Sigma}$ are $\Sigma_{n_1,n_2}=2(P-1)/(NP)$ for distinct indices $n_1,n_2\in[N]$, $\Sigma_{n,N+p}=\Sigma_{N+p,n}=-2/N$ for indices $n\in[N]$ and $p\in[P-1]$, and $\Sigma_{(N+p_1),(N+p_2)}=2/N$ for distinct indices $p_1,p_2\in[P-1]$.
\end{enumerate}
\end{Lemma}

\begin{Lemma}[Higher order covariance of Gaussians]
\label{lemma:covariance-gaussian-4th}
Let $(T_1,T_2)$ be jointly Gaussian random variables with zero mean, shared variance $\mathbb{E}[T_1^2]=\mathbb{E}[T_2^2]=\sigma^2$ and covariance $\mathbb{E}[T_1T_2]=\rho\sigma^2$. Then 
\begin{equation*}
\text{Cov}(T_1^4,T_2^4)=\mathbb{E}[T_1^4T_2^4]-\mathbb{E}[T_1^4]\mathbb{E}[T_2^4]=\sigma^8(72\rho^2+24\rho^4).
\end{equation*}
\end{Lemma}

\section{Single-Constraint Enumeration and Structural Comparison}
\label{sec:single-constraint}

We begin by enumerating $\mathscr{A}_1$ and $\mathscr{A}_2$ directly, and it is helpful here to provide some intuition. On the one hand, the row sum constraint of $\mathscr{A}_1$ is straightforward and involves just the row sums of $\mathbf{A}$. On the other hand, the paired column sum constraint of $\mathscr{A}_2$ involves both $\mathbf{A}$ and $\mathbf{X}$, effectively introducing an inequality constraint. To see why, fix a locus $p$, corresponding to the column pair $p$ and $(P+p)$. Using the definitions of $\Phi_{p,0}$ and $F_{p,0}$ (Eq.~\eqref{eq:Fp0}) and the fact that $0\leq X_{np}\leq 1$, it is clear that $0\leq \Phi_{p,0}\leq \sum_{n=1}^N(2-A_{np}-A_{n(P+p)})$, where the upper bound is just the number of zeroes appearing in columns $p$ and $(P+p)$ of $\mathbf{A}$. Hence, if the ancestry-specific allele dosages $(\Phi_{p,0},\Phi_{p,1})$ are equal to $(\phi_{p,0},\phi_{p,1})$, the number of zeroes appearing in columns $p$ and $(P+p)$ of $\mathbf{A}$ must be \emph{at least} $\phi_{p,0}$. Similar reasoning shows that the number of ones appearing in columns $p$ and $(P+p)$ of $\mathbf{A}$ must be at least $\phi_{p,1}$. Finally, the number of ones appearing in columns $p$ and $(P+p)$ of $\mathbf{X}$ is exactly $\sum_{n=1}^N(X_{np}+X_{n(P+p)})=\Phi_{p,0}+\Phi_{p,1}=\phi_{p,0}+\phi_{p,1}$.

\begin{Prop}[Size of $\mathscr{A}_1$]
\label{prop:a1}
For $N,P$ and row local ancestry tallies $(a_{1\cdot},\ldots,a_{N\cdot})$ fixed, the number of admixed arrays satisfying the constraint on row local ancestry tallies is
\begin{equation}
\left|\mathscr{A}_1(N,P;(a_{1\cdot},\ldots,a_{N\cdot}))\right|=\left(\prod_{n=1}^N{2P \choose a_{n\cdot}}\right)2^{2NP}.\label{eq:a1_exact}
\end{equation}
\end{Prop}
\begin{proof}
There are no constraints on $\mathbf{X}$, so there are $2^{2NP}$ possible choices of $\mathbf{X}$. For $\mathbf{A}$, there are ${2P\choose a_{n\cdot}}$ ways to assign ones or zeros along the $n$th row $\mathbf{A}_n$, and so there are $\prod_{n=1}^N{2P \choose a_{n\cdot}}$ possible choices of $\mathbf{A}$.
\end{proof}

To obtain the size of $\mathscr{A}_2$, we define a \emph{feasible set}, which encodes the implicit inequality constraint described earlier, imposed by $\bsphi_0$ and $\bsphi_1$ on the column sums of $\mathbf{A}$.

\begin{Def}[Feasible set]\label{def:feasible-set}
Let $\boldsymbol{\phi}_0=(\phi_{1,0},\ldots,\phi_{P,0})$ and $\boldsymbol{\phi}_1=(\phi_{1,1},\ldots,\phi_{P,1})$ be two integer-valued vectors. The feasible set $\mathscr{S}$ for the pair $(\boldsymbol{\phi}_0,\boldsymbol{\phi}_1)$ is defined as the set of vectors satisfying the following coordinate-wise inequalities with the elements of $(\boldsymbol{\phi}_0,\boldsymbol{\phi}_1)$.
\begin{equation}
\mathscr{S}=\{[\mathbf{v}_1,\mathbf{v}_2]\in([0,N]\cap \mathbb{Z})^{2P}:\mathbf{v}_1+\mathbf{v}_2\geq \bsphi_1,2N\mathbf{1}-(\mathbf{v}_1+\mathbf{v}_2)\geq \bsphi_0 \}.\label{eq:feasible-set}
\end{equation}
\end{Def}

\begin{Prop}[Size of $\mathscr{A}_2$]
\label{prop:a2}
For $N,P$ and ancestry-specific allele dosages $\bsphi_0$ and $\bsphi_1$ fixed, the number of admixed arrays satisfying the constraint on ancestry-specific allele dosages is
\begin{align}
\left|\mathscr{A}_2(N,P;\bsphi_0,\bsphi_1)\right| & =&\sum_{[\mathbf{s}_1,\mathbf{s}_2]\in \mathscr{S}} \left[\prod_{p=1}^P{N\choose s_{1p}}{N\choose s_{2p}}{s_{1p}+s_{2p}\choose \phi_{p,1}}{2N-(s_{1p}+s_{2p})\choose \phi_{p,0}}\right] \label{eq:a2_firstexp} \\
&=& \prod_{p=1}^P \left[{2N\choose \phi_{p,0}+\phi_{p,1}}{\phi_{p,0}+\phi_{p,1}\choose\phi_{p,1}}2^{2N-(\phi_{p,0}+\phi_{p,1})}\right],\label{eq:a2_secondexp}
\end{align}
where $\mathscr{S}$ is defined in Eq.~\eqref{eq:feasible-set}.
\end{Prop}
\begin{proof}
We first show Eq.~\eqref{eq:a2_firstexp} holds. As reasoned in the beginning of Section \ref{sec:single-constraint}, the vectors $\bsphi_0$ and $\bsphi_1$ constrain the paired column sums of $\mathbf{A}$ at each locus $p$ in the following way: letting $v_{1p}=\sum_{n=1}^N A_{np}$ and $v_{2p}=\sum_{n=1}^N A_{n(P+p)}$, then $v_{1p}+v_{2p}\geq \phi_{p,1}$ and $2N-(v_{1p}+v_{2p}) \geq \phi_{p,0}$. Additionally, $v_{1p},v_{2p}\in[0,N]$ because $A_{np}$ and $A_{n(P+p)}$ can only take on values $0$ or $1$. Collecting these inequalities across all $P$ loci, we obtain $\mathbf{v}_1+\mathbf{v}_2\geq \bsphi_1$, $2N\mathbf{1}-(\mathbf{v}_1+\mathbf{v}_2)\geq \bsphi_0$ and $[\mathbf{v}_1,\mathbf{v}_2]\in [0,N]^{2P}$, which defines $\mathscr{S}$.   

Picking an element $[\mathbf{s}_1,\mathbf{s}_2]$ from $\mathscr{S}$, we count the number of ways to populate the local ancestry matrix $\mathbf{A}$ with zeros and ones subject to the column sum constraints $\mathbf{s}_1$ and $\mathbf{s}_2$, followed by populating the allele dosage matrix $\mathbf{X}$ with zeros and ones such that the ancestry-specific allele dosages (paired column sums) are exactly $\bsphi_0$ and $\bsphi_1$. For each locus $p$, there are ${N\choose s_{1p}}{N\choose s_{2p}}$ ways to allocate zeros and ones to the $p$th and $(P+p)$th columns of $\mathbf{A}$. For the $p$th and $(P+p)$th columns of $\mathbf{X}$, there are ${s_{1p}+s_{2p}\choose \phi_{p,1}}{2N-(s_{1p}+s_{2p})\choose \phi_{p,0}}$ ways to allocate zeros and ones. To see why, imagine coloring each cell of the $p$th and $(P+p)$th columns of $\mathbf{X}$ blue or red, depending on whether the corresponding cell in the two columns of $\mathbf{A}$ was assigned zero or one. There are exactly $s_{1p}+s_{2p}$ red cells and $2N-(s_{1p}+s_{2p})$ blue cells. It remains to assign $\phi_{p,1}$ cells to have value $1$ amongst the red cells, and to assign $\phi_{p,0}$ cells amongst the blue cells to have value $1$. Therefore, by summing the counts of all admixed arrays $[\mathbf{A},\mathbf{X}]$ over each element of the feasible set $\mathscr{S}$, we obtain Eq.~\eqref{eq:a2_firstexp}. 

Now we show Eq.~\eqref{eq:a2_secondexp} holds, by counting the number of admixed arrays in a different way. First, we populate the allele dosage matrix $\mathbf{X}$ with ``colored'' ones, where the $p$th and $(P+p)$th columns will have exactly $\phi_{p,1}$ red ones and $\phi_{p,0}$ blue ones. There are ${2N\choose \phi_{p,0}+\phi_{p,1}}{\phi_{p,0}+\phi_{p,1}\choose\phi_{p,1}}$ ways to do this. Mapping the red cells to $1$ and blue cells to $0$ in the $p$th and $(P+p)$th columns of $\mathbf{A}$, there are $2N-(\phi_{p,0}+\phi_{p,1})$ cells that can be assigned values zero or one, without restriction. This can be done in $2^{2N-(\phi_{p,0}+\phi_{p,1})}$ ways. Therefore, by taking the product of these two quantities over all $p$ loci, we obtain Eq.~\eqref{eq:a2_secondexp}. 
\end{proof}

\subsection{Exact and Entropy Approximate Criteria}
\label{subsec:size-comparison}

It is not immediately obvious which of $|\mathscr{A}_{1}|$ and $|\mathscr{A}_{2}|$ is larger, so we obtain a necessary and sufficient criterion for $|\mathscr{A}_{1}|>|\mathscr{A}_{2}|$.

\begin{Th}[Necessary and sufficient criterion for $|\mathscr{A}_1|>|\mathscr{A}_2|$]
\label{thm:compare_a1_a2}
Suppose $N$ and $P$ are fixed, and the ancestry-specific allele dosages $\bsphi_0$ and $\bsphi_1$, as well as the local ancestry tallies $\bsa$, are given. For $n=1,\ldots,N$ let $\overline{a_{n\cdot}}=a_{n\cdot}/(2P)$; for $p=1,\ldots,P$ let $f_{p,0}=\phi_{p,0}/(2N)$ and $f_{p,1}=\phi_{p,1}/(2N)$ (not to be confused with $F_{p,0}$ and $F_{p,1}$ in Eqs.~\eqref{eq:Fp0} and \eqref{eq:Fp1}, which are normalized by a different denominator). Then $|\mathscr{A}_{1}|>|\mathscr{A}_{2}|$ if and only if the following inequality holds.
%\begin{align}
%\prod_{n=1}^N{2P\choose a_{n\cdot}} > 2^{-\sum_{p=1}^P(\phi_{p,0}+\phi_{p,1})} \left[\prod_{p=1}^P{2N\choose \phi_{p,0}+\phi_{p,1}}{\phi_{p,0}+\phi_{p,1}\choose \phi_{p,1}}\right].\label{eq:compare_a1_a2_exp}
%\end{align}
\begin{align}
\frac{1}{P}\sum_{p=1}^P(f_{p,0}+f_{p,1}) > \frac{1}{2NP}\left[\sum_{p=1}^P\left(\log{2N\choose 2N(f_{p,0}+f_{p,1})}+\log{2N(f_{p,0}+f_{p,1})\choose 2Nf_{p,1}}\right) - \right. \label{eq:compare_a1_a2_exp}\\ \left.\sum_{n=1}^N \log{2P\choose 2P\overline{a_{n\cdot}}}\right].\nonumber
\end{align}     
\end{Th}
\begin{proof}
By replacing $|\mathscr{A}_1|$ and $|\mathscr{A}_2|$ with Eqs.~\eqref{eq:a1_exact} and \eqref{eq:a2_secondexp}, and performing simplifications, we obtain the following equivalent inequalities.
\begin{eqnarray*}
|\mathscr{A}_1| > |\mathscr{A}_2| & \Longleftrightarrow & \prod_{n=1}^N{2P \choose a_{n\cdot}} > 2^{-\sum_{p=1}^P(\phi_{p,0}+\phi_{p,1})}\prod_{p=1}^P{2N\choose \phi_{p,0}+\phi_{p,1}}{\phi_{p,0}+\phi_{p,1}\choose \phi_{p,1}} \\
        & \Longleftrightarrow & \sum_{p=1}^P(\phi_{p,0}+\phi_{p,1}) > \sum_{p=1}^P\left(\log{2N\choose \phi_{p,0}+\phi_{p,1}}+\log{\phi_{p,0}+\phi_{p,1}\choose \phi_{p,1}}\right)-\sum_{n=1}^N\log{2P\choose a_{n\cdot}}.
\end{eqnarray*}
Plugging $\phi_{p,0}=2Nf_{p,0},\phi_{p,1}=2Nf_{p,1}$ and $a_{n\cdot}=2P\overline{a_{n\cdot}}$ into the last inequality, and dividing both sides by $2NP$, we recover Eq.~\eqref{eq:compare_a1_a2_exp}.
\end{proof}

Theorem \ref{thm:compare_a1_a2} applies to all parameterizations of marginal constraints and holds for all finite $N$ and $P$. However, if the marginal constraints are ``well-behaved'' in the sense that they do not lie on the boundary of the range of permissible values, then one may derive an approximately necessary and sufficient condition that is more insightful than Eq.~\eqref{eq:compare_a1_a2_exp}. 

To state and prove the approximate condition, we define some relevant quantities. Let $H(z_1,\ldots,z_I)\allowbreak =\sum_{i=1}^I z_i\log(1/z_i)$ denote the binary Shannon entropy function computed on a length $I$ vector $(z_1,\ldots,z_I)$ lying in the $(I-1)$-dimensional simplex for $I\geq 2$, and define 
\begin{equation}
\mathscr{T}(N,P)=\{(\bsa/(2P),\bsphi_0/(2N),\bsphi_1/(2N)))\in[0,1]^{N+2P}: |\mathscr{A}_1|>|\mathscr{A}_2|\},\label{eq:T_NP}
\end{equation}
the set of normalized marginal constraints for which the size of $\mathscr{A}_1$ exceeds the size of $\mathscr{A}_2$. Define the mean global ancestry entropy (row entropy, $H_1$), the mean ancestry-specific allele fraction entropy (column entropy, $H_2$), and the sum of mean ancestry-specific allele fractions ($\overline{f}$) by
\begin{eqnarray}
H_1(\bsa) & = &\frac{1}{N}\sum_{n=1}^NH\left(\overline{a_{n\cdot}},1-\overline{a_{n\cdot}}\right),\label{eq:ave_row_entropy} \\
H_2(\bsphi_0,\bsphi_1) & = &\frac{1}{P}\sum_{p=1}^PH\left(f_{p,0},f_{p,1},1-f_{p,0}-f_{p,1}\right),\label{eq:ave_col_entropy} \\
\overline{f}(\bsphi_0,\bsphi_1) & = & \frac{1}{P}\sum_{p=1}^P\left(f_{p,0}+f_{p,1}\right). \label{eq:f_bar}
\end{eqnarray}
Finally, define
\begin{align}
\widehat{\mathscr{T}(N,P)}=\Bigg\{(\bsa/(2P),\bsphi_0/(2N),\bsphi_1/(2N)))\in[0,1]^{N+2P}: H_1 - H_2 + \overline{f}>0\Bigg\}.\label{eq:T_hat_NP}
\end{align}
The following result shows that the difference between $\mathscr{T}(N,P)$ and $\widehat{\mathscr{T}(N,P)}$ is small for large $N$ and $P$, provided the marginal constraints are well-behaved.

%Assume that their limiting quantities $H_r^{(\infty,\infty)}=\lim_{N,P\to\infty}H_r^{(N,P)}$ and $H_c^{(\infty,\infty)}=\lim_{N,P\to\infty}H_c^{(N,P)}$ exist. 

\begin{Th}[Approximate criterion for $|\mathscr{A}_1|>|\mathscr{A}_2|$]
\label{thm:asymp_compare_a1_a2}
Suppose $N$ and $P$ are fixed, and assume that $\overline{a_{n\cdot}},f_{p,0}$ and $f_{p,1}$ are all contained in $(0,1)$ with $f_{p,0}+f_{p,1}< 1$ as $n$ ranges in $[N]$ and $p$ ranges in $[P]$. With $\mathscr{T}(N,P)$ and $\widehat{\mathscr{T}(N,P)}$ defined in Eqs.~\eqref{eq:T_NP} and \eqref{eq:T_hat_NP}, the set of marginal constraints that lie in their symmetric set difference, $\widehat{\mathscr{T}(N,P)}\triangle \mathscr{T}(N,P)$, is contained in the set of vectors $(\bsa/(2P),\bsphi_0/(2N),\bsphi_1/(2N)))$ for which the quantity $H_1-H_2 + \overline{f}$ lies in a subset of
\begin{align}
\left(-c\left(\frac{\log N}{N}+\frac{\log P}{P}\right),c\left(\frac{\log N}{N}+\frac{ \log P}{P}\right)\right],\label{eq:compare_a1_a2_asymp}
\end{align}
where $c$ is a positive constant that does not depend on $N$ or $P$.  
\end{Th}

\begin{proof}
We rely on a classical information-theoretic bound on the binomial coefficient (p.~309 of\linebreak \textcite{macwilliams1977theory}):
\begin{equation}
\sqrt{\frac{n}{8k(n-k)}} 2^{nH(k/n, 1-k/n)} \leq {n\choose k} \leq \sqrt{\frac{n}{2\pi k(n-k)}} 2^{nH(k/n,1-k/n)},\hspace{1cm}0<k<n.\label{eq:info-theoretic}
\end{equation}
Taking base-$2$ logarithms on Eq.~\eqref{eq:info-theoretic}, we obtain
\begin{equation}
\log{n\choose k} = nH\left(\frac{k}{n},1-\frac{k}{n}\right) + \frac{1}{2}\log n -\frac{1}{2}(\log k+\log(n-k)) - c, \label{eq:log-info-theoretic}
\end{equation}
where $c\in[\log(2\pi)/2,3/2]$. Applying Eq.\eqref{eq:log-info-theoretic} to each binomial coefficient appearing on the right hand side (RHS) of Eq.~\eqref{eq:compare_a1_a2_exp} yields
\begin{eqnarray*}
\log{2N\choose 2N(f_{p,0}+f_{p,1})} + \log{2N(f_{p,0}+f_{p,1})\choose 2Nf_{p,1}} & \in & \left[\tau_p-4,\tau_p-(2+\log\pi)\right],  \\
\log{2P \choose 2P\overline{a_{n\cdot}}} & \in & \left[\xi_n-2,\xi_n-\frac{1}{2}(2+\log\pi)\right],
\end{eqnarray*}
where the quantities $\tau_p$ and $\xi_n$ are defined as
\begin{eqnarray*}
\tau_p & = & 2N H(f_{p,0},f_{p,1},1-f_{p,0}-f_{p,1})-\frac{1}{2}\log N+\frac{1}{2}\log\left(\frac{1}{f_{p,0}f_{p,1}(1-f_{p,0}-f_{p,1})}\right), \\
\xi_n & = & 2PH(\overline{a_{n\cdot}},1-\overline{a_{n\cdot}})-\frac{1}{2}\log P + \frac{1}{2}\log\left(\frac{1}{\overline{a_{n\cdot}}(1-\overline{a_{n\cdot}})}\right).
\end{eqnarray*}
Thus, 
\begin{equation}
\frac{1}{2NP}\left[\sum_{p=1}^P \tau_p - \sum_{n=1}^N\xi_n\right]-\frac{4}{N} + \frac{2+\log\pi}{4P}\leq \text{RHS of Eq.~\eqref{eq:compare_a1_a2_exp}} \leq \frac{1}{2NP}\left[\sum_{p=1}^P \tau_p - \sum_{n=1}^N\xi_n\right]-\frac{2+\log\pi}{N} + \frac{1}{P},\label{eq:bound-rhs}
\end{equation}
with
\begin{equation*}
\begin{aligned}
\frac{1}{2NP}\Biggl[\sum_{p=1}^P \tau_p - \sum_{n=1}^N \xi_n\Biggr]
&= H_2 - H_1
 + \frac{1}{4}\Biggl(\frac{\log P}{P} - \frac{\log N}{N}\Biggr) \\
&\quad + \frac{1}{4NP}\Biggl[
\sum_{p=1}^P \log\!\Biggl(\frac{1}{f_{p,0}f_{p,1}\bigl(1-f_{p,0}-f_{p,1}\bigr)}\Biggr)
-\sum_{n=1}^N \log\!\Biggl(\frac{1}{\overline{a_{n\cdot}}\bigl(1-\overline{a_{n\cdot}}\bigr)}\Biggr)
\Biggr].
\end{aligned}
\end{equation*}
Define the following sets,
\begin{equation*}
\begin{aligned}
\widehat{\mathscr{T}(N,P)}^{\text{sub}}
&=\Bigg\{\bigl(\bsa/(2P),\bsphi_0/(2N),\bsphi_1/(2N)\bigr)\in[0,1]^{N+2P}:\;
H_1 - H_2 + \overline{f} > \\
&\qquad \frac{1}{4NP}\Biggl[
\sum_{p=1}^P \log\!\Biggl(\frac{1}{f_{p,0}f_{p,1}\bigl(1-f_{p,0}-f_{p,1}\bigr)}\Biggr)
-\sum_{n=1}^N \log\!\Biggl(\frac{1}{\overline{a_{n\cdot}}\bigl(1-\overline{a_{n\cdot}}\bigr)}\Biggr)
\Biggr] \\
&\qquad + \left(\frac{\log P}{4P} -\frac{\log N}{4N}\right)
+\frac{1}{P}-\frac{2+\log\pi}{N}
\Bigg\}, \\[0.5em]
\widehat{\mathscr{T}(N,P)}^{\text{sup}}
&=\Bigg\{\bigl(\bsa/(2P),\bsphi_0/(2N),\bsphi_1/(2N)\bigr)\in[0,1]^{N+2P}:\;
H_1 - H_2 + \overline{f} > \\
&\qquad \frac{1}{4NP}\Biggl[
\sum_{p=1}^P \log\!\Biggl(\frac{1}{f_{p,0}f_{p,1}\bigl(1-f_{p,0}-f_{p,1}\bigr)}\Biggr)
-\sum_{n=1}^N \log\!\Biggl(\frac{1}{\overline{a_{n\cdot}}\bigl(1-\overline{a_{n\cdot}}\bigr)}\Biggr)
\Biggr] \\
&\qquad + \left(\frac{\log P}{4P} -\frac{\log N}{4N}\right)
+\frac{2+\log\pi}{4P}-\frac{4}{N}
\Bigg\}.
\end{aligned}
\end{equation*}

From Eq.~\eqref{eq:bound-rhs} we get $\hat{\mathscr{T}}^{\text{sub}}\subset \mathscr{T}\subset\hat{\mathscr{T}}^{\text{sup}}$, where $\mathscr{T}$ is defined in Eq.~\eqref{eq:T_NP}. Moreover, the quantity $\sum_{p=1}^P\log(1/[f_{p,0}f_{p,1}(1-f_{p,0}-f_{p,1})])$ is bounded between $P\cdot\min_p \{\log(1/[f_{p,0}f_{p,1}(1-f_{p,0}-f_{p,1})])\}$ and $P\cdot\max_p \{\log(1/[f_{p,0}f_{p,1}(1-f_{p,0}-f_{p,1})])\}$, which are both finite since $f_{p,0},f_{p,1}$ and $f_{p,0}+f_{p,1}$ lie in $(0,1)$. This shows that $\sum_{p=1}^P\log(1/[f_{p,0}f_{p,1}(1-f_{p,0}-f_{p,1})])=O(P)$. Similar reasoning shows that $\sum_{n=1}^N\log(1/[\overline{a_{n\cdot}}(1-\overline{a_{n\cdot}})])=O(N)$. From set-theoretic arguments, we have 
\begin{equation*}
\mathscr{T}\triangle\hat{\mathscr{T}} =(\mathscr{T}\setminus\hat{\mathscr{T}})\cup(\hat{\mathscr{T}}\setminus\mathscr{T})\subset (\hat{\mathscr{T}}^\text{sup}\setminus\hat{\mathscr{T}})\cup(\hat{\mathscr{T}}\setminus\hat{\mathscr{T}}^\text{sub}).
\end{equation*}
Thus, $\mathscr{T}\triangle\hat{\mathscr{T}}$ is contained in the union of two sets, defined by the collection of vectors $(\overline{a_{1\cdot}},\ldots,\overline{a_{N\cdot}},f_{1,0},\ldots,\linebreak  f_{P,0},f_{1,1},\ldots,f_{P,1})=(\bsa/(2P),\bsphi_0/(2N),\bsphi_1/(2N)))\in[0,1]^{N+2P}$ satisfying
\begin{align*}
H_1 - H_2 + \overline{f} \in \left(O\left(\frac{1}{N}+\frac{1}{P}\right)+\left(\frac{\log P}{4P} -\frac{\log N}{4N}\right)+\frac{2+\log\pi}{4P}-\frac{4}{N},0\right]\cup\\
\left(0,O\left(\frac{1}{N}+\frac{1}{P}\right)+\left(\frac{\log P}{4P} -\frac{\log N}{4N}\right)+\frac{1}{P}-\frac{2+\log\pi}{N}\right].
\end{align*}
The non-zero endpoints of both intervals are dominated by $\log(N)/N+\log(P)/P$, so there exists a positive constant $c$ such that their union is a subset of the interval $(-c(\log(N)/N+\log(P)/P),c(\log(N)/N+\log(P)/P)]$.
\end{proof}

It is clear from Eq.~\eqref{eq:compare_a1_a2_asymp} that as $N,P\to\infty$, the difference between $\mathscr{T}$ and $\hat{\mathscr{T}}$ approaches the empty set. Therefore, Theorem \ref{thm:asymp_compare_a1_a2} establishes the sign of the quantity $H_1-H_2+\overline{f}$, which can be computed directly from the constraints $\bsa,\bsphi_0$ and $\bsphi_1$, as an approximate criterion for large $N$ and large $P$. This quantity has an information-theoretic interpretation: between the sets $\mathscr{A}_1$ and $\mathscr{A}_2$, $\mathscr{A}_1$ is larger if and only if the entropy across the row constraints exceeds the entropy across the column constraints minus the sum of ancestry-specific fractions in the allele dosage matrix.   

We state a special case whereby the marginal constraints are uniform, or ``semi-regular'', which facilitates visualization of our exact and approximate criteria. Specifically, we let $\bsa=(a,\ldots,a)$ for some constant $a\in\{1\ldots,2P-1\}$, and $\bsphi_0=(\phi_0,\ldots,\phi_0)$ and $\bsphi_1=(\phi_1,\ldots,\phi_1)$ for some constants $\phi_0,\phi_1\in \{1,\ldots,2N-1\}$ subject to the constraint $0< \phi_0+\phi_1< 2N$. We similarly define the normalized marginal constraints $\overline{a}=a/(2P), f_0=\phi_0/(2N)$ and $f_1=\phi_1/(2N)$, which also lie in $(0,1)$ but no longer depend on the indices $n$ and $p$. We can now restate Theorems \ref{thm:compare_a1_a2} and \ref{thm:asymp_compare_a1_a2} for this special case. 

\begin{Cor}[Semi-regular criteria for $|\mathscr{A}_1|>|\mathscr{A}_2|$]
\label{cor:compare_a1_a2_uniform}
Suppose $N$ and $P$ are fixed and the marginal constraints are uniform, with ancestry-specific allele dosages all equal to $\phi_0$ and $\phi_1$, as well as the local ancestry tallies all equal to $a$. Let $f_0,f_1$ and $\overline{a}$ denote the normalized marginal constraints. Then $|\mathscr{A}_{1}|>|\mathscr{A}_{2}|$ if and only if the following inequality holds.
\begin{align}
\frac{1}{2P}\log{2P\choose 2P\overline{a}}-\frac{1}{2N}\left[\log{2N\choose 2N\left(f_0+f_1\right)} + \log{2N\left(f_0+f_1\right)\choose 2Nf_1}\right] + (f_0+f_1) > 0.\label{eq:compare_a1_a2_uniform}
\end{align}
Assume that the corresponding large $N,P$ limits --- $f_0^{(\infty,\infty)}=\lim_{N,P\to\infty} f_0, f_1^{(\infty,\infty)}=\lim_{N,P\to\infty} f_1$ and $\overline{a}^{(\infty,\infty)}=\lim_{N,P\to\infty}\overline{a}$ --- exist, and that $f_0^{(\infty,\infty)},f_1^{(\infty,\infty)},f_0^{(\infty,\infty)}+f_1^{(\infty,\infty)}$ and $\overline{a}^{(\infty,\infty)}$ lie in $(0,1)$. Then, as $N,P\to\infty$, $|\mathscr{A}_1|>|\mathscr{A}_2|$ if and only if the following inequality holds.
\begin{align}
H\left(\overline{a}^{(\infty,\infty)},1-\overline{a}^{(\infty,\infty)}\right) - H\left(f_0^{(\infty,\infty)},f_1^{(\infty,\infty)},1-f_0^{(\infty,\infty)}-f_1^{(\infty,\infty)}\right)+\left(f_0^{(\infty,\infty)}+f_1^{(\infty,\infty)}\right)>0,\label{eq:largeNlargeP_a_1_a2_uniform}
\end{align}
where $H$ denotes the binary Shannon entropy function. 
\end{Cor}

\begin{figure}[ht]
\begin{center}
\includegraphics[width=\textwidth]{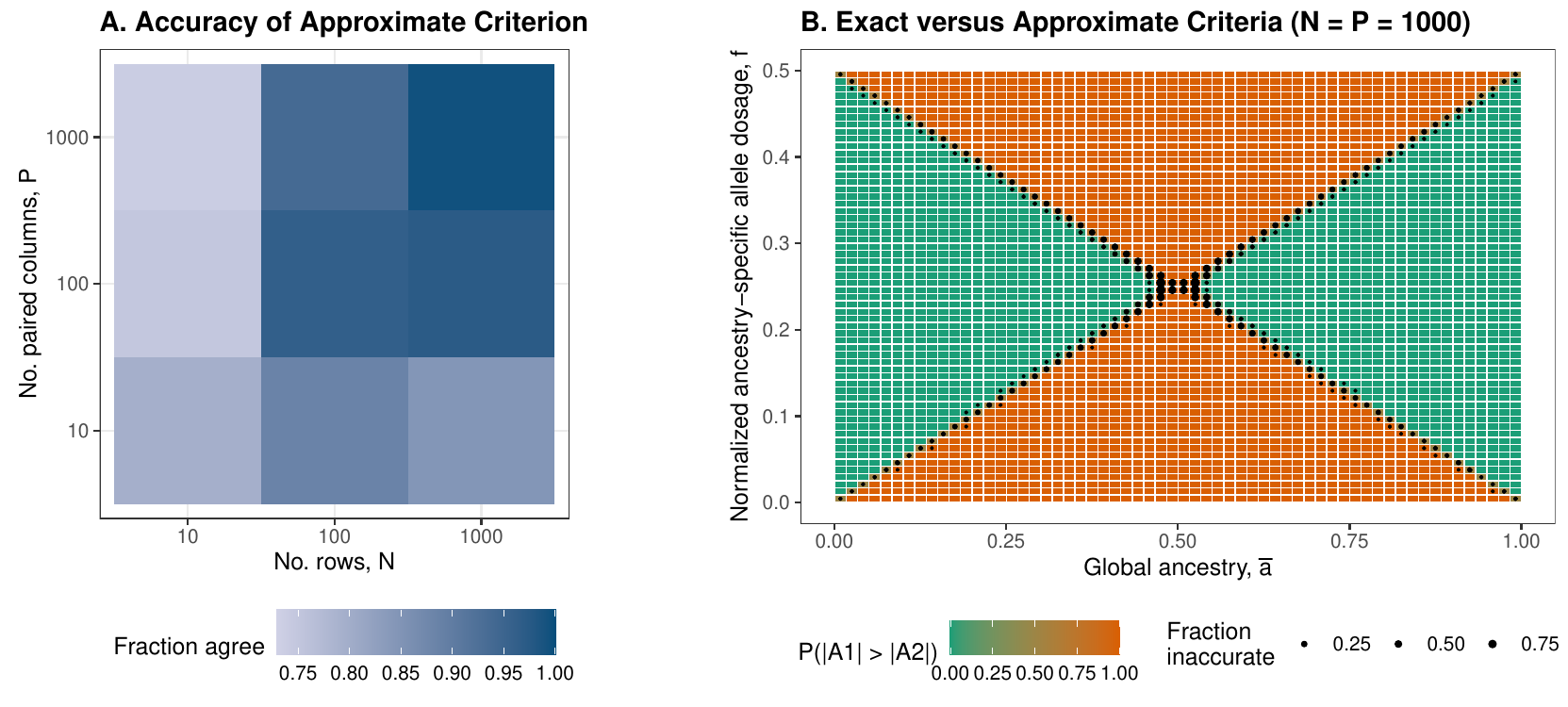}
\end{center}
\caption{Visualizing Corollary \ref{cor:compare_a1_a2_uniform}. \textbf{A.} Heat map of the fraction of predictions obtained from applying the entropy approximate criterion (Eq.~\eqref{eq:largeNlargeP_a_1_a2_uniform}) that agree with the true classification obtained by direct calculation of $|\mathscr{A}_1|$ and $|\mathscr{A}_2|$. Fractions range from $0.74$ ($(N,P)=(10,10^3)$) to $0.99$ ($(N,P)=(10^3,10^3)$). \textbf{B.} Heat map showing ranges of values of the pair of quantities $(\overline{a},f)$ for which $\mathscr{A}_1$ is larger than $\mathscr{A}_2$, when $N=P=1000$ and $\overline{a}$ and $f$ are allowed to range within their discrete domains. To improve visualization, the roughly $2\times 10^6$ choices of $(\overline{a},f)$ are partitioned into $3,600$ square bins, with each bin reporting the fraction of $(\overline{a},f)$ pairs --- denoted $P(|\mathscr{A}_1|>|\mathscr{A]_2|})$ --- for which $|\mathscr{A}_1| > |\mathscr{A}_2|$. Black points identify bins for which the approximate bound (Eq.~\eqref{eq:largeNlargeP_a_1_a2_uniform}) predicted the larger set inaccurately, with larger points indicating a greater fraction of inaccurate predictions within the bin.}
\label{fig:entropy_approx}
\end{figure}

To visualize Corollary \ref{cor:compare_a1_a2_uniform}, we set $f_0=f_1=f$ and allow $f$ and $\overline{a}$ to range over their discrete domains, $\{j/2N:j=1,\ldots,N-1\}$ and $\{i/2P:i=1,\ldots,2P-1\}$ respectively. We then use Eq.~\eqref{eq:compare_a1_a2_uniform} (ground truth) and Eq.~\eqref{eq:largeNlargeP_a_1_a2_uniform} (approximation) to determine whether $\mathscr{A}_1$ is larger than $\mathscr{A}_2$ for each pair $(\overline{a},f)$. Figure \ref{fig:entropy_approx} shows that the fraction of points for which the approximation is inaccurate (i.e., Eq.~\eqref{eq:largeNlargeP_a_1_a2_uniform} classifying $\mathscr{A}_1$ as larger when in fact $\mathscr{A}_2$ is, according to Eq.~\eqref{eq:compare_a1_a2_uniform}) approaches $0$, with a \emph{small but modest} error fraction of $0.01$ for $N=P=10^3$ (Figure \ref{fig:entropy_approx}A). Inaccurate predictions of the larger set occur along the diagonals of the rectangular domain $[0,1]\times [0,0.5]$ in which the quantities $(\overline{a},f)$ range, corresponding to instances where the quantities on both sides of the bound of Eq.~\eqref{eq:largeNlargeP_a_1_a2_uniform} are numerically very close (Figure \ref{fig:entropy_approx}B). This last observation reflects Theorem \ref{thm:asymp_compare_a1_a2}, which states that the region of inaccurate prediction lies in a set for which the difference between the entropy expressions is very small (Eq.~\eqref{eq:compare_a1_a2_asymp}). To explain the first observation, we next show that the fraction of points lying in the symmetric difference $\widehat{\mathscr{T}(N,P)}\triangle \mathscr{T}(N,P)$ --- as described in Theorem \ref{thm:asymp_compare_a1_a2} --- \emph{decays moderately}, behaving like $\sqrt{\log(N)/N + \log(P)/P}$. To obtain this result, we approximate the discrete domains by their continuous counterparts and drop the constant $c$ appearing in Eq.~\eqref{eq:compare_a1_a2_asymp}. Specifically, let $\mathscr{B}_0\subset [0,1]^3$ denote the set of points $(\overline{a},f_0,f_1)$ for which $1>\overline{a},f_0,f_1>0$ and $f_0+f_1<1$, and let $\mathscr{B}\subset \mathscr{B}_0$ denote the subset of points for which $H\left(\overline{a},1-\overline{a}\right) - H\left(f_0,f_1,1-f_0-f_1\right)+\left(f_0+f_1\right)\in \left(-(\log(N)/N+\log(P)/P),\log(N)/N+\log(P)/P\right]$. Then $\text{Vol}(\mathscr{B})/\text{Vol}(\mathscr{B}_0)$ is an approximate upper bound of the fraction of points lying in $\widehat{\mathscr{T}(N,P)}\triangle \mathscr{T}(N,P)$. This approximation is reasonable, because the uniform measure on a bounded domain of equally spaced points converges weakly to the Lebesgue measure over the domain as the spacing approaches $0$. %Finally, our proof is particular to the semi-regular case only to avoid having the essential arguments be buried under notational excess; it generalizes to all parameterizations of the constraints.                 

\begin{Th}[Semi-regular classification error fraction]
\label{thm:convergence_uniform}
Let $\mathscr{B}_0$ and $\mathscr{B}$ be defined as in the preceding paragraph. Then, there exist universal positive constants $K_1$ and $K_2$ such that the following inequality holds. 
\begin{align}
\frac{\text{Vol}(\mathscr{B})}{\text{Vol}(\mathscr{B}_0)} \leq K_1\sqrt{\frac{\log N}{N} + \frac{\log P}{P}} + K_2 \left(\frac{\log N}{N} + \frac{\log P}{P}\right)^{3/2}.\label{eq:rate-of-convergence}
\end{align}
Because the right-hand side expression is bounded above by a constant multiple of $\sqrt{\log(N)/N + \log(P)/P}$, the decay rate is also bounded above by a quantity of order $\sqrt{\log(N)/N + \log(P)/P}$.
\end{Th}

\begin{proof}
Denote by $I_{\mathscr{B}}$ the indicator function for whether a triple $(\overline{a},f_0,f_1)$ lies in $\mathscr{B}$. Denote by $\mathscr{F}\subset [0,1]^2$ the set of points $(f_0,f_1)$ satisfying $0<f_0,f_1<1$ with $f_0+f_1<1$; this is just an isosceles right triangle with unit side length. 

First compute $\text{Vol}(\mathscr{B}_0)$; this is the volume of a unit cube sliced along the face diagonal, which is $0.5$. Next compute $\text{Vol}(\mathscr{B})$. By Fubini's theorem, this can be computed by evaluating the following double integral, 
\begin{equation}
\text{Vol}(\mathscr{B})=\int_{\mathscr{F}} \left(\int_0^1 I_{\mathscr{B}} d\overline{a}\right) df_0df_1. \label{appeq:vol-B}
\end{equation}
We evaluate Eq.~\eqref{appeq:vol-B} as follows. For each fixed pair $(f_0,f_1)\in\mathscr{F}$ we define the $(f_0,f_1)$-conditioned set $\mathscr{B}(f_0,f_1)\subset [0,1]$ as the set of all $\overline{a}$ satisfying 
\begin{equation}
    H\left(\overline{a},1-\overline{a}\right) - H\left(f_0,f_1,1-f_0-f_1\right)+\left(f_0+f_1\right)\in \left(-\left(\frac{\log N}{N}+\frac{\log P}{P}\right),\frac{\log N}{N} + \frac{\log P}{P}\right]. \label{appeq:bad-set}
\end{equation} 
We obtain upper bounds on the length of $\mathscr{B}(f_0,f_1)$ by controlling the gradient of the entropy function $H(x)=x\log(1/x)+(1-x)\log(1/(1-x))$, allowing us to replace the inner integral in Eq.~\eqref{appeq:vol-B} with the upper bound quantities. Finally, we integrate over $\mathscr{F}$ to obtain an upper bound on $\text{Vol}(\mathscr{B})$. Critically, the quantity $1/H'(x)$ (one over the gradient of entropy) that helps with obtaining upper bounds is itself unbounded as $x\rightarrow 1/2$, so we rely on $\varepsilon$ regularization \cite{tao2010epsilon}. 

Concretely, let $\varepsilon>0$, and further assume that $f_0$ and $f_1$ are given with $\theta=\theta(f_0,f_1)=H\left(f_0,f_1,1-\right.\allowbreak \left. f_0-f_1\right)-\left(f_0+f_1\right)$. $\mathscr{B}(f_0,f_1)$ is exactly the set of all $\overline{a}$ such that $H(\overline{a})\in (\theta-(\log(N)/N + \log(P)/P),\theta+(\log(N)/N + \log(P)/P)]$. We consider two cases: (1) $\theta+(\log(N)/N+\log(P)/P) \leq 1-2\varepsilon$; and (2) $\theta+(\log(N)/N+\log(P)/P) > 1-2\varepsilon$. \\

\noindent\textit{Case 1}. The binary entropy function $H(x)$ is non-negative and symmetric about $x=1/2$, where its value is also uniquely maximized with $H(1/2)=1$. It is also strictly increasing on $[0,1/2]$ and strictly decreasing on $(1/2,1]$. Thus, in Case 1, the set of all $\overline{a}$ is the disjoint union of two intervals that are symmetric about $1/2$, where the left interval is strictly contained in $[0,1/2]$. By bounding the left interval from above, we immediately obtain an upper bound for $\mathscr{B}(f_0,f_1)$ by multiplying the first bound by $2$. Concretely, let $\Delta^\ast$ denote the left interval just described. Let $\gamma_U$ be the smaller root of the equation $H(x)=\theta+\log(N)/N + \log(P)/P$; note $\gamma_U\in[0,1/2)$. Because $H(x)$ is concave, the tangent line to $H(x)$ at $x=\gamma_U$ --- namely, $y=H'(\gamma_U)(x-\gamma_U)+H(\gamma_U)$ --- lies above its graph, so $\Delta^\ast$ can be approximated by the interval $\overline{\Delta}$ with right endpoint $\gamma_U$ and left endpoint given by $x_L=[\theta-(\log(N)/N + \log(P)/P) - H(\gamma_U)]/H'(\gamma_U) + \gamma_U$ (see Figure \ref{fig:convergence_rate_proof}A). The length of $\overline{\Delta}$ is $\gamma_U-x_L=-[\theta-(\log(N)/N + \log(P)/P) - H(\gamma_U)]/H'(\gamma_U) = 2(\log(N)/N + \log(P)/P)/H'(\gamma_U)$. Because $H(x)$ is strictly concave, $1/H'(x)$ is strictly increasing on $[0,1/2]$. Let $H^{-1}(1-2\varepsilon)$ denote the smaller root of $H(x)=1-2\varepsilon$. Because $H$ is also strictly increasing on $[0,0.5]$, $\gamma_U\leq H^{-1}(1-2\varepsilon)$ and hence $1/H'(\gamma_U) \leq 1/H'(H^{-1}(1-2\varepsilon))$. We now approximate $1/H'(H^{-1}(1-2\varepsilon))$. By performing a Taylor expansion of $H(x)$ around $x=1/2$, we have $H(x)\approx 1 - 2\left(x-1/2\right)^2/\ln2$. Equating the Taylor approximation on the right with $1-2\varepsilon$, we obtain $x\approx 1/2 - \sqrt{\varepsilon \ln2}$. Thus, $1/H'(H^{-1}(1-2\varepsilon))\approx 1/H'(1/2-\sqrt{\varepsilon\ln2})=\ln2/\ln\left((1/2+\sqrt{\varepsilon\ln2})/(1/2-\sqrt{\varepsilon\ln2})\right)$, and so we obtain an upper bound on $\overline{\Delta}$:      
\begin{equation*}
\overline{\Delta} \leq \frac{2\left(\frac{\log N}{N} + \frac{\log P}{P}\right)}{H'(H^{-1}(1-2\varepsilon))} \approx \frac{2\ln2\left(\frac{\log N}{N} + \frac{\log P}{P}\right)}{\ln\left(\frac{\frac{1}{2}+\sqrt{\varepsilon\ln2}}{\frac{1}{2}-\sqrt{\varepsilon\ln2}}\right)}.
\end{equation*}
Therefore, $\mathscr{B}(f_0,f_1)$ is bounded above as follows.
\begin{equation}
|\mathscr{B}(f_0,f_1)|= 2\Delta^\ast \leq 2\overline{\Delta} \leq \frac{4\ln2\left(\frac{\log N}{N} + \frac{\log P}{P}\right)}{\ln\left(\frac{\frac{1}{2}+\sqrt{\varepsilon\ln2}}{\frac{1}{2}-\sqrt{\varepsilon\ln2}}\right)}.\label{eq:case-1}
\end{equation}

\begin{figure}[ht]
\begin{center}
\includegraphics[width=0.9\textwidth]{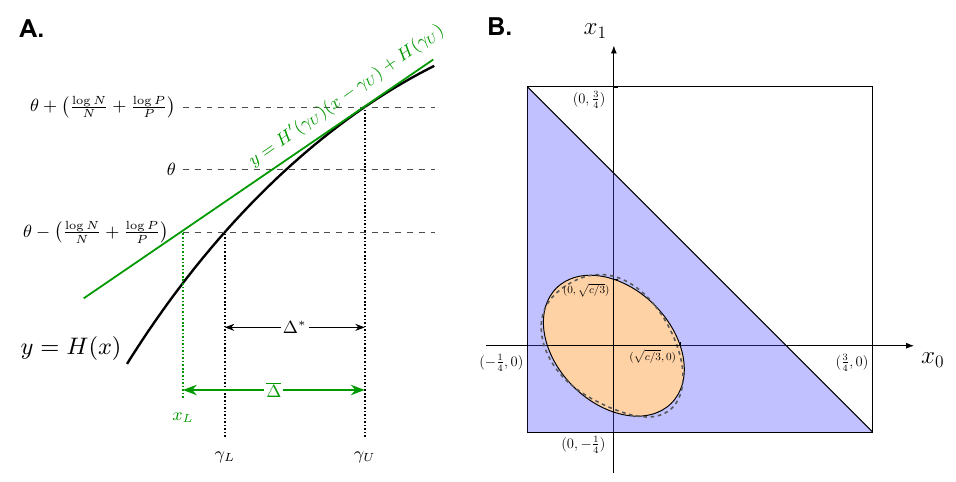}
\end{center}
\caption{Visualizations for the proof of Theorem \ref{thm:convergence_uniform}. \textbf{A.} Zoomed in graph of the binary entropy function $H(x)=x\log x+(1-x)\log(1-x)$ restricted to $x\in[0,0.5]$. \textbf{B.} Translation of the graph of the set of points $(f_0,f_1)$ that represent ancestry-specific allele fraction constraints, where the parameterization $x_0=f_0-1/4$ and $x_1=f_1-1/4$ is used. The quantity $c=(2\ln2)\varepsilon+\ln2(\log N/N +\log P/P)$, as defined in the proof of Theorem \ref{thm:convergence_uniform}. The ellipse and the region outside it bounded by the isosceles triangle are colored differently, as they depict two cases for which the volume of $\mathscr{B}$ (defined in Main Text) is to be bounded from above. The dashed closed curve overlaying the boundary of the ellipse depicts the actual region for one of the cases, which we approximate to second order with the ellipse.}
\label{fig:convergence_rate_proof}
\end{figure}

\noindent\textit{Case 2}. Depending on the value of $\theta$, $\mathscr{B}(f_0,f_1)$ is either a disjoint union of two intervals symmetric about $1/2$, or a single interval centered at $1/2$. In either case, $\mathscr{B}(f_0,f_1)$ is no larger than the interval $[\gamma_L,1-\gamma_L]$ with length $1-2\gamma_L$, where $\gamma_L$ is defined as the smaller root of the equation $H(x)=\theta-(\log(N)/N+\log(P)/P)$ (see Figure \ref{fig:convergence_rate_proof}A). Let $H^{-1}(1-2\varepsilon-2(\log(N)/N + \log(P)/P))$ denote the smaller root of $H(x)=1-2\varepsilon-2(\log(N)/N + \log(P)/P)$. Because $H$ is strictly increasing on $[0,0.5]$, $\gamma_L \geq H^{-1}(1-2\varepsilon-2(\log(N)/N + \log(P)/P))$, and so $|\mathscr{B}(f_0,f_1)| \leq 1-2H^{-1}(1-2\varepsilon-2(\log(N)/N + \log(P)/P))$. We now approximate the quantity on the right. Performing a Taylor expansion of $H(x)$ around $x=1/2$ and equating the approximation with $1-2\varepsilon-2(\log(N)/N + \log(P)/P)$, we obtain $H^{-1}(1-2\varepsilon-2(\log(N)/N + \log(P)/P))\approx 1/2 - \sqrt{\ln2\left(\varepsilon + \log(N)/N + \log(P)/P\right)}$. Therefore, $\mathscr{B}(f_0,f_1)$ is bounded above as follows.
\begin{equation}
|\mathscr{B}(f_0,f_1)| \leq 1-2H^{-1}\left(1-2\varepsilon-2\left(\frac{\log N}{N} + \frac{\log P}{P}\right)\right)\approx 2\sqrt{\ln2\left(\varepsilon + \frac{\log N}{N}+\frac{\log P}{P}\right)}. \label{eq:case-2}
\end{equation}

To finish bounding $\text{Vol}(\mathscr{B})$ from above, we need the volumes of the set of points $(f_0,f_1)$ that satisfy each case. Concretely, 
\begin{eqnarray*}
\mathscr{D}_1 & = & \left\{(f_0,f_1):\theta(f_0,f_1) \leq 1-2\varepsilon - \left(\frac{\log N}{N} + \frac{\log P}{P}\right)\right\}, \\
\mathscr{D}_2 & = & \left\{(f_0,f_1):\theta(f_0,f_1) > 1-2\varepsilon - \left(\frac{\log N}{N} + \frac{\log P}{P}\right)\right\},
\end{eqnarray*}
where $\mathscr{D}_1$ and $\mathscr{D}_2$ denote the disjoint sets satisfying Case 1 and Case 2 respectively (note $\mathscr{D}_1\cup \mathscr{D}_2=\mathscr{F}$). By Lemma \ref{lemma:theta-function}, $\mathscr{D}_2$ is contained in some ball containing $(1/4,1/4)$. To approximate $\mathscr{D}_2$, we perform a Taylor expansion of $\theta(f_0,f_1)$ around $(1/4,1/4)$, obtaining 
\begin{align*}
\theta(f_0,f_1) \approx\;& \theta\!\left(\tfrac14,\tfrac14\right)
+ \sum_{i\in\{0,1\}}
\left.\frac{\partial\theta}{\partial f_i}\right|_{(\tfrac14,\tfrac14)}
\left(f_i-\tfrac14\right) \\
&\quad
+ \frac12\sum_{k=0}^2
\left.\frac{\partial^2 \theta}{(\partial f_0)^k (\partial f_1)^{2-k}}\right|_{(\tfrac14,\tfrac14)}
\left(f_0-\tfrac14\right)^k \left(f_1-\tfrac14\right)^{2-k} \\
=\;& 1-\frac{1}{\ln2}\left[
\left(f_0+f_1-\tfrac12\right)^2
+2\left(f_0-\tfrac14\right)^2
+2\left(f_1-\tfrac14\right)^2
\right].
\end{align*}
Replacing $\theta(f_0,f_1)$ in the definition of $\mathscr{D}_2$ with the last quantity above, we obtain as an approximation to $\mathscr{D}_2$ the set of points $(f_0,f_1)\in \mathscr{F}$ satisfying 
\begin{equation}
\left(f_0+f_1-\frac{1}{2}\right)^2 + 2\left(f_0-\frac{1}{4}\right)^2+2\left(f_1-\frac{1}{4}\right)^2 > c,\label{eq:f0f1-approx-region}
\end{equation}
where $c=(2\ln2)\varepsilon + \ln2(\log(N)/N + \log(P)/P)$. To compute the volume of this set, reparameterize Eq.~\eqref{eq:f0f1-approx-region} by setting $x_0=f_0-1/4$ and $x_1=f_1-1/4$. This yields the constrained region 
\[
\left\{
\begin{aligned}
-\tfrac14 \le x_0, x_1 \le \tfrac34,\\
x_0 + x_1 \le \tfrac12,\\
3x_0^{2} + 2x_0x_1 + 3x_1^{2} > c,
\end{aligned}
\right.
\]
which is the interior of a rotated ellipse of the form $Ax_0^2+Bx_0x_1+Cx_1^2=D$, whose volume is given by $2\pi D/\sqrt{4AC-B^2}$ (Figure \ref{fig:convergence_rate_proof}B). Setting $A=C=3,B=2$ and $D=c$, we obtain 
\begin{equation}
\text{Vol}(\mathscr{D}_2)\approx \left(\frac{\pi\ln2}{\sqrt{2}}\right)\varepsilon + \frac{\ln2}{4\sqrt{2}}\left(\frac{\log N}{N} + \frac{\log P}{P}\right).\label{eq:vol-D2}
\end{equation}
Consequently, 
\begin{equation}
\text{Vol}(\mathscr{D}_1)\approx \frac{1}{2}-\left[\left(\frac{\pi\ln2}{\sqrt{2}}\right)\varepsilon + \frac{\ln2}{4\sqrt{2}}\left(\frac{\log N}{N} + \frac{\log P}{P}\right)\right] \leq \frac{1}{2}.\label{eq:vol-D1}
\end{equation}
Putting the pieces together (Eqs.~\eqref{eq:case-1}, \eqref{eq:case-2}, \eqref{eq:vol-D2} and \eqref{eq:vol-D1}), we obtain the following bound on $\text{Vol}(\mathscr{B})$.
\begin{align*}
\mathrm{Vol}(\mathscr{B})
&= \int_{\mathscr{F}} \left(\int_0^1 I_{\mathscr{B}}\, d\overline{a}\right)\, df_0\, df_1 \\
&\le \int_{\mathscr{D}_1}
\left[
\frac{4\ln2\left(\frac{\log N}{N} + \frac{\log P}{P}\right)}
{\ln\!\left(\frac{\frac12+\sqrt{\varepsilon\ln2}}{\frac12-\sqrt{\varepsilon\ln2}}\right)}
\right] df_0\, df_1
+ \int_{\mathscr{D}_2}
\left(2\sqrt{\ln2\left(\varepsilon + \frac{\log N}{N}+\frac{\log P}{P}\right)}\right) df_0\, df_1 \\
&\le \frac{2\ln2\left(\frac{\log N}{N} + \frac{\log P}{P}\right)}
{\ln\!\left(\frac{\frac12+\sqrt{\varepsilon\ln2}}{\frac12-\sqrt{\varepsilon\ln2}}\right)} \\
&\quad
\hspace{1cm}+ \left[\left(\sqrt{2}\ln2\right)\pi\varepsilon
+ \frac{\ln2}{2\sqrt{2}}\left(\frac{\log N}{N} + \frac{\log P}{P}\right)\right]
\sqrt{\ln2\left(\varepsilon + \frac{\log N}{N}+\frac{\log P}{P}\right)}.
\end{align*}
Since the inequality above holds for any (small) positive $\varepsilon$, choose $\varepsilon=\log(N)/N + \log(P)/P$. The second term of the upper bound simplifies to $k_2\left(\log(N)/N + \log(P)/P\right)^{3/2}$ where $k_2=(\ln2)^{3/2}\left(2\pi+1/2\right)\linebreak\approx 3.91$. To simplify the first term, we apply Lemma \ref{lemma:ratio-approx} with $\varepsilon=\log(N)/N + \log(P)/P$, obtaining the expression $k_1\sqrt{\log(N)/N + \log(P)/P}$ where $k_1=\sqrt{\ln 2} \approx 0.833$. Therefore Eq.~\eqref{eq:rate-of-convergence} holds with $K_1=2k_1$ and $K_2=2k_2$, completing the proof.
\end{proof}

\section{Doubly Constrained Admixed Arrays}
\label{sec:results}

Having analyzed the impact of the row-sum and paired column-sum constraints separately, and identified the entropic structure underlying each, we now turn to the doubly constrained family, $\mathscr{A}_{12}$. In the derivation of Proposition \ref{prop:a2}, Eq.~\eqref{eq:a2_firstexp} was obtained by counting the number of arrays conditioned on assigning a feasible vector of column sums to the local ancestry matrix $\mathbf{A}$. This conditioning strategy extends directly to the doubly constrained family $\mathscr{A}_{12}$. Accordingly, let $\mathscr{A}(\mathbf{v}_1,\mathbf{v}_2,\bsa)$ be the set of $N\times 2P$ binary matrices whose column sum is $[\mathbf{v}_1,\mathbf{v}_2]$ (both $\mathbf{v}_1$ and $\mathbf{v}_2$ are integer vectors lying in $[0,N]^{P}$) and whose row sum is $\bsa=(a_{1\cdot},\ldots,a_{N\cdot})$. 

\begin{Prop}[Size of $\mathscr{A}_{12}$]
\label{prop:a12}
For given $N,P$, and $\bsa=(a_{1\cdot},\ldots,a_{N\cdot}),\bsphi_0$ and $\bsphi_1$ fixed, the number of admixed arrays satisfying these constraints on row local ancestry tallies and ancestry-specific allele dosages is
\begin{equation}
\left|\mathscr{A}_{12}\right|=\sum_{[\mathbf{s}_1,\mathbf{s}_2]\in\mathscr{S}} \left[|\mathscr{A}(\mathbf{s}_1,\mathbf{s}_2,\bsa)|\left(\prod_{p=1}^P{s_{1p}+s_{2p}\choose \phi_{p,1}}{2N-(s_{1p}+s_{2p})\choose \phi_{p,0}}\right)\right],\label{eq:a12_exact}
\end{equation}
where $\mathscr{S}$ is defined in Eq.~\eqref{eq:feasible-set}, and $\mathscr{A}(\mathbf{s}_1,\mathbf{s}_2,\bsa)$ is the set of $N\times 2P$ binary matrices whose column and row sums are $[\mathbf{s}_1,\mathbf{s}_2]$ and $\bsa$ respectively. 
\end{Prop}
\begin{proof}
Suppose that a choice of $\bsa$ leads to a non-empty set $\mathscr{A}(\mathbf{s}_1,\mathbf{s}_2,\bsa)$. For each local ancestry matrix $\mathbf{A}\in\mathscr{A}(\mathbf{s}_1,\mathbf{s}_2,\bsa)$, the accompanying allele dosage matrix $\mathbf{X}$ shares the same constraint on the numbers of zeros and ones as in Proposition \ref{prop:a2}. Thus, the number of such $\mathbf{X}$'s is $\prod_{p=1}^P{s_{1p}+s_{2p}\choose \phi_{p,1}}{2N-(s_{1p}+s_{2p})\choose \phi_{p,0}}$, completing the proof. 
\end{proof}

To get a sense of the number of summands contributing to the sum in Eq.~\eqref{eq:a12_exact}, we quantify the size of the feasible set $\mathscr{S}$ below.

\begin{Prop}[Size of feasible set]
\label{lemma:feasible-set-size}
Let $\mathscr{S}$ be the feasible set defined in Proposition \ref{prop:a2}, which consists of all local ancestry dosages satisfying the inequality constraints imposed by $\bsphi_0=(\phi_{p,0}:p=1,\ldots P)$ and $\bsphi_1=(\phi_{p,1}:p=1,\ldots P)$. Then the size of $\mathscr{S}$ is given by the following quantity. 
\begin{equation}
|\mathscr{S}|= \prod_{p=1}^P\left[\sum_{\ell=\phi_{p,1}}^{2N-\phi_{p,0}}\left(\min\{\ell,2N-\ell\}+1\right)\right]. \label{eq:feasible-set-size}
\end{equation} 
\end{Prop}

\begin{proof}[Proof of Lemma \ref{lemma:feasible-set-size}]
In the definition of $\mathscr{S}$, the condition is that at any locus $p\in[P]$, the $p$th entries $(v_{1p},v_{2p})\in\{0,\ldots,N\}^2$ are restricted to sum to $\ell$, where $\ell$ ranges from $\phi_{p,1}$ to $2N-\phi_{p,0}$. For each fixed $\ell$ within the range, there are exactly $\min\{\ell,2N-\ell\}+1$ possible pairs of non-negative integers $(v_{1p},v_{2p})$ bounded above by $N$ such that $v_{1p}+v_{2p}=\ell$. Thus, the inner sum of Eq.~\eqref{eq:feasible-set-size} enumerates the number of possible pairs for locus $p$, and we obtain the size of $\mathscr{S}$ by multiplying the inner sum across all $P$ loci, which is just the product in Eq.~\eqref{eq:feasible-set-size}.   
\end{proof}

\subsection{Asymptotic Enumeration in a Semi-regular Case}
\label{subsec:asymptotic}

{\red Enumerating $\mathscr{A}_{12}$ is generally not straightforward. First, the feasible set $\mathscr{S}$ is typically very large: if $\phi_{p,0}=\phi_{p,1}=K$ for some positive integer $K$ at most $N$, then Proposition \ref{lemma:feasible-set-size} shows that $|\mathscr{S}|=\left[(N+1)^2-K(K+1)\right]^P$, which is generally $\Theta(N^{2P})$ unless $K/N=1-o(1)$. Second, as seen in the expression of $|\mathscr{A}_{12}|$ in Proposition \ref{prop:a12}, each term in the sum involves a $P$-fold product, complicating direct approximation strategies. Critically, even in the general semi-regular setting, where the normalized marginal constraints are controlled by just three parameters $\overline{a},f_0$ and $f_1$, a non-trivial argument is needed to obtain the generating function associated with $|\mathscr{A}_{12}|$. The Hessian associated with the Taylor expansion around the saddle point has a different rank depending on the region of parameter space. For these reasons, we analyze the doubly constrained model in a ``semi-regular $1/2$'' setting, in which the inequality constraints become equalities with the generating function admitting a tractable analytic representation, and so the different independence heuristic arises succinctly. The case-by-case asymptotics in the general semi-regular setting across all feasible triplets $(\overline{a},f_0,f_1)$ will be reported in a separate paper.}

%Enumerating $\mathscr{A}_{12}$ is challenging for several reasons. First, the feasible set $\mathscr{S}$ is typically very large: if $\phi_{p,0}=\phi_{p,1}=K$ for some positive integer $K$ at most $N$, then Proposition \ref{lemma:feasible-set-size} shows that $|\mathscr{S}|=\left[(N+1)^2-K(K+1)\right]^P$, which is generally $\Theta(N^{2P})$ unless $K/N=1-o(1)$. Second, as seen in the expression of $|\mathscr{A}_{12}|$ in Proposition \ref{prop:a12}, each term in the sum involves a $P$-fold product, complicating direct approximation strategies. Critically, even in the general semi-regular setting, the generating function associated with $|\mathscr{A}_{12}|$ does not admit a straightforward saddle-point formulation. After several attempts, we found that the constraints require enlarging the summation domain from $\mathscr{S}$ to $([0,N]\cap \mathbb{Z})^{2P}$, introducing additional dependencies that obstruct direct asymptotic analysis. For these reasons, we analyze the doubly constrained model in a ``semi-regular $1/2$'' setting, in which the inequality constraints become equalities and the generating function admits a tractable analytic representation.

\begin{Th}[Approximation of $\log|\mathscr{A}_{12}|$ in the semi-regular $1/2$ case]
\label{thm:growth-rate-a12}
Suppose that all ancestry-specific allele dosages are $N$ and all row local ancestry tallies are $P$. That is, $\phi_{1,0}=\ldots=\phi_{P,0}=\phi_{1,1}=\ldots=\phi_{P,1}=N$ and $a_{1\cdot}=\ldots=a_{N\cdot}=P$. Then the following upper bound holds: 
\begin{equation}
\alpha_{12}=\log|\mathscr{A}_{12}| \leq 2NP -\frac{1}{2}\left[N\log(\pi P) + P\log(\pi N)\right] + \frac{1}{2}\log(\pi NP).\label{eq:a12_upper_bound}
\end{equation}
Moreover, for $N,P\to\infty$ with $N=\Theta(P)$, 
\begin{equation}
\alpha_{12}=\log|\mathscr{A}_{12}| = 2NP -\frac{1}{2}\left[N\log(\pi P) + P\log(\pi N)\right] + \frac{1}{2}\log(\pi NP)-\frac{(N+P-1)^2}{8NP}\log(e)+O\left(\frac{1}{\sqrt{m}}\right),\label{eq:a12_approx_semiregular}
\end{equation}
where $m=\min\{N,P\}$ and $e:=\lim_{n\to\infty}(1+1/n)^n\approx2.71828$.
\end{Th}
\begin{proof}
We proceed along three key steps. First, we obtain a Cauchy integral expression for $|\mathscr{A}_{12}|$. Next, we simplify the Cauchy integral. Finally, we estimate the simplified integral using a saddle-point approximation. Our approximation quickly leads to the upper bound Eq.~\eqref{eq:a12_upper_bound}, and with some extra work, also to Eq.~\eqref{eq:a12_approx_semiregular}. 

\underline{Step 1: Obtaining the Cauchy integral.} We first observe that the constraints placed on $\phi_{p,0}$ and $\phi_{p,1}$ force the $p$th and $(P+p)$th columns of the local ancestry matrix $\mathbf{A}$ to have exactly $N$ entries of each ancestry (i.e., $N$ colored red and the other $N$ colored blue, following the visualization of Figure \ref{fig:example}). This follows from plugging $\phi_{p,0}=\phi_{p,1}=N$ into the inequalities defining all feasible local ancestry dosages, which simplifies to all local ancestry dosage vectors $[\mathbf{s}_1,\mathbf{s}_2]$ in the feasible set $\mathscr{S}$ satisfying $\mathbf{s}_1+\mathbf{s}_2=N\mathbf{1}$. Moreover, conditioned on assigning the ancestries for the entries of these matrices, the $p$th and $(P+p)$th columns of $\mathbf{X}$ must consist of all ones. Thus, taking into consideration the additional constraint placed on $a_{n\cdot}$, each admixed array $[\mathbf{A},\mathbf{X}]$ included in the total count is such that $\mathbf{X}$ is the matrix of all ones, while $\mathbf{A}$ must have $N$ entries of each ancestry between the $p$th and $(P+p)$th columns and exactly $P$ entries of each ancestry along the $n$th row. The problem is thus reduced to enumerating all such $\mathbf{A}$. Let us represent $\mathbf{A}$ as a collection of $NP$ \emph{cell pairs}. Specifically, let $\mathbf{v}_{n,p}=(A_{np},A_{n(P+p)})\in\{(0,0),(0,1),(1,0),(1,1)\}$ be the cell pair for the $n$th row at the $p$th locus, where $1\leq n\leq N$ and $1\leq p\leq P$. We claim that the constraints placed on $\mathbf{A}$ are equivalent to the collection of cell pairs $\mathbf{v}_{n,p}$ satisfying the condition that the occurrences of $(0,0)$ and $(1,1)$ are equal across rows $n$ and across columns $p$. To see this, observe that the constraints on $\mathbf{A}$ are equivalent to 
\begin{eqnarray*}
(\forall p\in[P])(\exists a_p,b_p\in\mathbb{N}) & \mathbf{v}_{1,p}+\ldots+\mathbf{v}_{N,p}=(a_p,b_p)~\wedge~a_p+b_p=N, \\
(\forall n\in[N])(\exists c_n,d_n\in\mathbb{N}) & \mathbf{v}_{n,1}+\ldots+\mathbf{v}_{n,P}=(c_n,d_n)~\wedge~c_n+d_n=P. 
\end{eqnarray*}
Now, $(a_p,b_p)=\kappa_p^{00}(0,0)+\kappa_p^{01}(0,1)+\kappa_p^{10}(1,0)+\kappa_p^{11}(1,1)$ and $(c_n,d_n)=\kappa_n^{00}(0,0)+\kappa_n^{01}(0,1)+\kappa_n^{10}(1,0)+\kappa_n^{11}(1,1)$, where $(\kappa_p^{00},\kappa_p^{01},\kappa_p^{10},\kappa_p^{11})$ and $(\kappa_n^{00},\kappa_n^{01},\kappa_n^{10},\kappa_n^{11})$ track the number of cell pairs of each type contributing to the $p$th locus and the $n$th row respectively. By directly comparing the entries and invoking $a_p+b_p=N$ for the $p$th locus, we obtain $\kappa_p^{00}=\kappa_p^{11}$. By a similar comparison for the $n$th individual, we obtain $\kappa_n^{00}=\kappa_n^{11}$. Therefore the claim is verified.

In the next step of obtaining the Cauchy integral, we construct a Laurent polynomial generating function for which the number of row paired cell collections $\{\mathbf{v}_{n,p}\}$ satisfying $\kappa_p^{00}=\kappa_p^{11}$ and $\kappa_n^{00}=\kappa_n^{11}$ for all $n$ and $p$ corresponds to a coefficient. Introduce the variables $\{x_n\}$, $\{w_n\}$ and $\{z_p\}$ to track the row sums, row balances and column balances respectively. For each row paired cell $\mathbf{v}_{n,p}$, define their weight by $f(x_n,w_n,z_p)=x_n^2w_nz_p+w_n^{-1}z_p^{-1}+2x_n$. Collecting the variables as $\mathbf{x}=(x_1,\ldots,x_N)$, $\mathbf{w}=(w_1,\ldots,w_N)$ and $\mathbf{z}=(z_1,\ldots,z_P)$, define
\begin{equation*}
F(\mathbf{x},\mathbf{w},\mathbf{z})=\prod_{n=1}^N\prod_{p=1}^Pf(x_n,w_n,z_p).%\label{eq:laurent_series}
\end{equation*}
We claim that the term whose coefficient corresponds to $|\mathscr{A}_{12}|$ is $x_1^P\cdots x_N^P$. To see this, first note that the row sum constraint of $P$ on the $n$th individual immediately implies that the power of $x_n$ in our term of interest is $P$. Second, note that the balance constraints ($\kappa_p^{00}=\kappa_p^{11}$ and $\kappa_n^{00}=\kappa_n^{11}$) implies that the powers of $w_n$ and $z_p$ are zero. Therefore, $|\mathscr{A}_{12}|=[x_1^P\cdots x_N^P][w_1^0\cdots w_N^0][z_1^0\cdots z_P^0] F(\mathbf{x},\mathbf{w},\mathbf{z})$. By Cauchy's theorem, 
\begin{equation}
|\mathscr{A}_{12}|=\frac{1}{(2\pi i)^{2N+P}} \oint\cdots\oint \frac{F(\mathbf{x},\mathbf{w},\mathbf{z})}{x_1^{P+1}\cdots x_N^{P+1}\cdot w_1\cdots w_N\cdot z_1\cdots z_P} dx_1\cdots dx_N\cdot dw_1\cdots dw_N\cdot dz_1\cdots dz_P,\label{eq:cauchy_integral}
\end{equation}
where each contour of the integral in Eq.~\eqref{eq:cauchy_integral} can be set to a counterclockwise path along the unit circle. 

\underline{Step 2: Simplifying the Cauchy integral.} To simplify Eq.~\eqref{eq:cauchy_integral}, we reparameterize the variables $x_n,w_n$ and $z_p$ with polar coordinates. For $n\in[N]$ and $p\in[P]$, let $x_n=\exp(i\vartheta_n),w_n=\exp(i\alpha_n)$ and $z_p=\exp(i\beta_p)$, with $\vartheta_n,\alpha_n$ and $\beta_p$ lying in $[-\pi,\pi]$. Then $dx_n/d\vartheta_n=ix_n,dw_n/d\alpha_n=iw_n$ and $dz_p/d\beta_p=iz_p$, so the Cauchy integral simplifies (we use shorthand notation $d\bsvartheta=d\vartheta_1\cdots d\vartheta_N$, $d\bsalpha=d\alpha_1\cdots d\alpha_N$ and $d\bsbeta=d\beta_1\cdots d\beta_P$): 
\begin{align}
|\mathscr{A}_{12}|
&= \frac{1}{(2\pi i)^{2N+P}}
\int_{[-\pi,\pi]^{2N+P}}
\Biggl[
\prod_{n=1}^N\prod_{p=1}^P
\Bigl(
e^{i(2\vartheta_n+\alpha_n+\beta_p)}
+ e^{-i(\alpha_n+\beta_p)}
+ 2e^{i\vartheta_n}
\Bigr)
\Biggr] \notag\\
&\qquad\qquad \cdot
\frac{i^{2N+P}}{e^{iP(\vartheta_1+\ldots+\vartheta_N)}}
\, d\bsvartheta\, d\bsalpha\, d\bsbeta \nonumber\\
    & = \frac{1}{(2\pi)^{2N+P}} \int_{[-\pi,\pi]^{2N+P}}\prod_{n=1}^N\prod_{p=1}^P 4\cos^2\left(\frac{\vartheta_n+\alpha_n+\beta_p}{2}\right) d\bsvartheta\, d\bsalpha\, d\bsbeta, & \label{eq:reparam_integral}
\end{align}
where the last equality follows from standard identities of Euler and trigonometry, $e^{ix}=\cos x + i\sin x$ and $\cos(2x)= 2\cos^2x - 1$. Observing that the quantity in Eq.~\eqref{eq:reparam_integral} depends on $\vartheta_n$ and $\alpha_n$ only through their sum $\vartheta_n+\alpha_n$ allows us to further simplify Eq.~\eqref{eq:reparam_integral}. Concretely, for $n\in [N]$ let $u_n=\vartheta_n+\alpha_n~(\text{mod}~2\pi)$. Let $h(t)=4\cos^2(t/2)$, and rewrite Eq.~\eqref{eq:reparam_integral} as $1/(2\pi)^{2N+P} \int_{[-\pi,\pi]^{2N+P}}\prod_{n=1}^N\prod_{p=1}^P h(\vartheta_n+\alpha_n+\beta_p)d\bsvartheta d\bsalpha d\bsbeta$. We claim that 
\begin{equation}
\int_{[-\pi,\pi]^{2N+P}} \prod_{n=1}^N\prod_{p=1}^P h(\vartheta_n+\alpha_n+\beta_p) d\bsvartheta\, d\bsalpha\, d\bsbeta = (2\pi)^N\int_{[-\pi,\pi]^{N+P}} \prod_{n=1}^N\prod_{p=1}^P h(u_n+\beta_p) d\bsu\, d\bsbeta. \label{eq:reduced_integral}
\end{equation}
To prove our claim, we first apply Fubini's theorem to recognize that we can hold the $\beta_p$ fixed and evaluate a $2N$-dimensional inner integral in $\vartheta_n$ and $\alpha_n$. Thus, fix $\beta_1,\ldots,\beta_P$. For each $n\in[N]$, the contribution of $\theta_n,\alpha_n$ to the integrand $\prod_{n=1}^N\prod_{p=1}^P h(\vartheta_n+\alpha_n+\beta_p)$ is exactly $\prod_{p=1}^P h(\vartheta_n+\alpha_n+\beta_p)$. Now define a map $K$ from the torus $\mathscr{G}^2=(\mathbb{R}/2\pi\mathbb{Z})^2$ to itself, given by $K(\vartheta,\alpha)=(\vartheta+\alpha,\alpha)~(\text{mod}~2\pi)$. This is a smooth bijection with inverse $K^{-1}(u,v)=(u-v,v)~(\text{mod}~2\pi)$. Its derivative matrix is 
\begin{equation*}
D_K=\begin{bmatrix}
\partial u/\partial\vartheta & \partial u/\partial\alpha \\
\partial v/\partial\vartheta & \partial v/\partial\alpha
\end{bmatrix} = \begin{bmatrix}
1 & 1 \\
0 & 1
\end{bmatrix},
\end{equation*}
which has $\det(D_K)=1$. Hence, Haar measure on $\mathscr{G}^2$ is invariant under $K$, and so for any function $g$ integrable on $\mathscr{G}^2$,
\begin{equation}
\int_{\mathscr{G}^2} g(\vartheta,\alpha)d\vartheta d\alpha = \int_{\mathscr{G}^2} g(K^{-1}(u,v))du dv.\label{eq:change-variables}
\end{equation}
Let $\Phi(u)=\prod_{p=1}^Ph(u+\beta_p)$; $\Phi$ is continuous on $[-\pi,\pi]$ and is $2\pi$-periodic. Set $g(\vartheta,\alpha)=\Phi(\vartheta+\alpha)$ in Eq.~\eqref{eq:change-variables}, so that $g(K^{-1}(u,v))=g(u-v,v)=\Phi(u)$. We obtain for each $n$
\begin{eqnarray*}
\int_{[-\pi,\pi]^2} \prod_{p=1}^Ph(\vartheta_n+\alpha_n+\beta_p)d\vartheta_n d\alpha_n & = & \int_{\mathscr{G}^2} g(\vartheta_n,\alpha_n)d\vartheta_n d\alpha_n \\
    & = & \int_{\mathscr{G}^2} \Phi(u_n)du_ndv_n \\
    & = & \left(\int_{-\pi}^\pi dv_n\right)\left(\int_{-\pi}^\pi \prod_{p=1}^P h(u_n+\beta_p)du_n\right) \\
    & = & (2\pi)\int_{-\pi}^\pi \prod_{p=1}^P h(u_n+\beta_p)du_n.
\end{eqnarray*}
Multiplying across all $n\in[N]$, we obtain an equality between the original $2N$-dimensional inner integral and a $N$-dimensional integral: 
\begin{equation}
\int_{[-\pi,\pi]^{2N}} \prod_{n=1}^N\prod_{p=1}^P h(\vartheta_n+\alpha_n+\beta_p)d\bsvartheta d\bsalpha = (2\pi)^N\int_{[-\pi,\pi]^N} \prod_{n=1}^N\prod_{p=1}^P h(u_n+\beta_p)d\bsu,\label{eq:reduced_inner_integral}
\end{equation}
where we use the shorthand notation $d\bsu= du_1\cdots du_N$. Integrating both sides of Eq.~\eqref{eq:reduced_inner_integral} over the $\beta_p$'s, we obtain Eq.~\eqref{eq:reduced_integral} as desired.

Next, we show that an extra $2\pi$ factor can be removed from the right-hand-side expression in Eq.~\eqref{eq:reduced_integral}. Concretely, we claim that
\begin{equation}
\int_{[-\pi,\pi]^{N+P}} \prod_{n=1}^N\prod_{p=1}^Ph(u_n+\beta_p) d\bsu d\bsbeta = (2\pi) \int_{[-\pi,\pi]^{N+P-1}}\prod_{n=1}^N\left[h(v_n)\prod_{p=1}^{P-1} h(v_n+b_p)\right]d\bsv d\bsb,\label{eq:doubly_reduced_integral}
\end{equation}
where we use the shorthand notation $d\bsv = dv_1\cdots dv_N$ and $d\bsb = db_1\cdots db_{P-1}$. Our claim is proved in a similar manner to how we justified Eq.~\eqref{eq:reduced_integral}. Let $\mathscr{G}^{N+P}=(\mathbb{R}/2\pi\mathbb{Z})^{N+P}$ be the $(N+P)$-torus, and define the map $K:\mathscr{G}^{N+P}\rightarrow \mathscr{G}^{N+P}$, given by $K(u_1,\ldots,u_N,\beta_1,\ldots,\beta_P)=(v_1,\ldots,v_N,b_1,\ldots,b_{P-1},t)$, where $v_n=u_n+t$ for $n\in[N]$, $b_p=\beta_p-t$ for $p\leq P-1$ and $t=\beta_P$. All expressions are defined modulo $2\pi$. This is again a smooth bijection with inverse $K^{-1}(v_1,\ldots,v_N,b_1,\allowbreak\ldots,b_{P-1},t)=(v_1-t,\ldots,v_N-t,b_1+t,\ldots,b_{P-1}+t,t)$. Its derivative matrix is 
\begin{equation*}
D_K=\begin{bmatrix}
\mathbf{I}_N & \mathbf{0}_{N\times (P-1)} & \mathbf{1}_{N\times 1} \\
\mathbf{0}_{(P-1)\times N} & \mathbf{I}_{P-1} & -\mathbf{1}_{(P-1)\times 1} \\
\mathbf{0}_{1\times N} & \mathbf{0}_{1\times (P-1)} & 1
\end{bmatrix},
\end{equation*}
which is upper triangular with ones along the diagonal and thus satisfies $\det(D_K)=1$. Hence, for any function $g$ integrable on $\mathscr{G}^{N+P}$, 
\begin{equation}
\int_{\mathscr{G}^{N+P}} g(u_1,\ldots,u_N,\beta_1,\ldots,\beta_P)d\bsu d\bsbeta = \int_{\mathscr{G}^{N+P}} g(K^{-1}(v_1,\ldots,v_N,b_1,\ldots,b_{P-1},t))d\bsv d\bsb dt.\label{eq:change-variables-2}
\end{equation}
Now let $g(u_1,\ldots,u_N,\beta_1,\ldots,\beta_P)=\prod_{n=1}^N\prod_{p=1}^P h(u_n+\beta_p)$. Under the inverse map $K^{-1}$, 
\begin{equation*}
u_n+\beta_p=\begin{cases}
    (v_n-t) + (b_p+t) = v_n+b_p & \text{if}~p\leq P-1 \\
    v_n & \text{if}~p=P.
\end{cases}
\end{equation*}
So $g(K^{-1}(v_1,\ldots,v_N,b_1,\ldots,b_{P-1},t))=\prod_{n=1}^N\left[h(v_n)\prod_{p=1}^{P-1} h(v_n+b_p)\right]$, which does not depend on $t$. Plugging this choice of $g$ into Eq.~\eqref{eq:change-variables-2}, we obtain
\begin{eqnarray*}
\int_{[-\pi,\pi]^{N+P}} \prod_{n=1}^N\prod_{p=1}^Ph(u_n+\beta_p) d\bsu d\bsbeta & = &  \int_{\mathscr{G}^{N+P}} g~d\bsu d\bsbeta \\
    & = & \int_{\mathscr{G}^{N+P}} g\circ K^{-1}~d\bsv d\bsb dt \\
    & = & \left(\int_{-\pi}^{\pi}dt\right)\left(\int_{[-\pi,\pi]^{N+P-1}}\prod_{n=1}^N\left[h(v_n)\prod_{p=1}^{P-1} h(v_n+b_p)\right]d\bsv d\bsb\right) \\
    & = & (2\pi)\int_{[-\pi,\pi]^{N+P-1}}\prod_{n=1}^N\left[h(v_n)\prod_{p=1}^{P-1} h(v_n+b_p)\right]d\bsv d\bsb,
\end{eqnarray*}
as desired. 

Therefore, applying Eqs.~\eqref{eq:reduced_integral} and \eqref{eq:doubly_reduced_integral} to Eq.~\eqref{eq:cauchy_integral}, we obtain the following simplified integral representation of $|\mathscr{A}_{12}|$:
\begin{equation}
|\mathscr{A}_{12}| = \frac{1}{(2\pi)^{N+P-1}}\int_{[-\pi,\pi]^{N+P-1}}\prod_{n=1}^N\left[h(v_n)\prod_{p=1}^{P-1} h(v_n+b_p)\right]d\bsv d\bsb,\label{eq:a12_reduced_integral_form}
\end{equation}
where $h(t)=4\cos^2(t/2)$.

\underline{Step 3: Approximating the integral.} This step is more involved, but can be organized around two key stages. We first rewrite Eq.~\eqref{eq:a12_reduced_integral_form} in terms of a standard saddle-point approximation, with extra terms distributed according to a truncated Gaussian law. We then apply concentration inequalities and hypercontractivity to recover the 4th moment explicitly and show that the higher order terms are $O(1/\sqrt{m})$. 

Concretely, write the integrand in Eq.~\eqref{eq:a12_reduced_integral_form} as
\begin{eqnarray}
\prod_{n=1}^N\left[h(v_n)\prod_{p=1}^{P-1} h(v_n+b_p)\right] &=& \exp\left(\sum_{n=1}^N\left(\ln(h(v_n))+\sum_{p=1}^{P-1}\ln(h(v_n+b_p))\right)\right) \nonumber\\
& = & \exp\left(\sum_{n=1}^N\sum_{p=1}^P \ln h(t_{np})\right) \label{eq:integrand}
\end{eqnarray}
where 
\begin{equation*}
t_{np} = \begin{cases}
    v_n+b_p & \text{if}~p\leq P-1 \\
    v_n & \text{if}~p=P
\end{cases}.
\end{equation*}
A Taylor expansion of $\ln(h(t))$ around $t=0$ gives $\ln(h(t))=\ln4 -t^2/4-t^4/96+O(t^6)$, so the natural candidate for a saddle-point approximation is given by the quadratic form
\begin{equation}
Q(\mathbf{v},\mathbf{b})= \frac{1}{4}\sum_{n=1}^N\sum_{p=1}^Pt_{np}^2.\label{eq:Q-quad-form}
\end{equation}
Define the weight
\begin{equation}
W(\mathbf{v},\mathbf{b})=\exp\left(\sum_{n=1}^N\sum_{p=1}^P \left[\ln\left(\frac{h(t_{np})}{4}\right)+\frac{t_{np}^2}{4}\right]\right),\label{eq:W-diff}
\end{equation}
where we use the shorthand $\mathbf{v}=(v_1,\ldots,v_N)$ and $\mathbf{b}=(b_1,\ldots,b_{P-1})$, so that the right hand side of Eq.~\eqref{eq:integrand} can be expressed as 
\begin{equation*}
\exp\left(\sum_{n=1}^N\sum_{p=1}^P \ln h(t_{np})\right)=4^{NP}\exp(-Q(\mathbf{v},\mathbf{b}))W(\mathbf{v},\mathbf{b}).
\end{equation*}
Now $W(\mathbf{v},\mathbf{b})\in[0,1]$ for all $(\mathbf{v},\mathbf{b})\in[-\pi,\pi]^{N+P-1}$. This is because each summand in Eq.~\eqref{eq:W-diff} is non-negative, by Lemma \ref{lemma:quadratic-upper-bound}. This implies the following upper bound on $|\mathscr{A}_{12}|$:
\begin{eqnarray*}
|\mathscr{A}_{12}| & = & \frac{1}{(2\pi)^{N+P-1}}\int_{[-\pi,\pi]^{N+P-1}} 4^{NP}\exp(-Q(\mathbf{v},\mathbf{b}))W(\mathbf{v},\mathbf{b})d\bsv d\bsb \\
    & \leq & \frac{4^{NP}}{(2\pi)^{N+P-1}}\int_{[-\pi,\pi]^{N+P-1}} \exp(-Q(\mathbf{v},\mathbf{b}))d\bsv d\bsb \\
    & \leq & \frac{4^{NP}}{(2\pi)^{N+P-1}}\int_{\mathbb{R}^{N+P-1}} \exp(-Q(\mathbf{v},\mathbf{b}))d\bsv d\bsb.
\end{eqnarray*}
To evaluate the last expression above, we observe that $Q(\mathbf{v},\mathbf{b}) = \frac{1}{2}(\mathbf{v},\mathbf{b})\left(\mathbf{B}/2\right)(\mathbf{v},\mathbf{b})^T$, where 
\begin{equation}
\mathbf{B}=\begin{bmatrix}
P\mathbf{I}_N & \mathbf{1}_{N\times (P-1)} \\
\mathbf{1}_{(P-1)\times N} & N\mathbf{I}_{P-1}
\end{bmatrix}.\label{eq:B-matrix}
\end{equation}
where $\mathbf{1}_{k\times \ell}$ denotes the $k\times \ell$ matrix of all ones. By Lemma \ref{lemma:covariance-matrix-properties}, $\det(\mathbf{B})=N^{P-1}P^{N-1}$, so the Gaussian integral is just $Z=\int_{\mathbb{R}^{N+P-1}}\exp(-Q)d\bsv\bsb = (2\pi)^{(N+P-1)/2}/\sqrt{\det(\mathbf{B}/2)}=(4\pi)^{(N+P-1)/2}/\linebreak (N^{(P-1)/2}P^{(N-1)/2})$. Plugging $Z$ back into the upper bound for $|\mathscr{A}_{12}|$, we obtain
\begin{equation*}
|\mathscr{A}_{12}|\leq \frac{4^{NP}N^{(P-1)/2}P^{(N-1)/2}}{\pi^{(N+P-1)/2}} \Longrightarrow \alpha_{12}\leq 2NP-\frac{1}{2}\left[N\log(\pi P)+P\log(\pi N)\right] + \frac{1}{2}\log(\pi NP),
\end{equation*}
the latter of which is Eq.~\eqref{eq:a12_upper_bound}. 

We now derive the approximation when $N=\Theta(P)$. Recall that $m=\min\{N,P\}$. Fix $\varepsilon\in(0,1/12)$, and set $\delta=m^{-1/2+\varepsilon}$. Define the region 
\begin{equation}
\mathscr{R}=\{(\mathbf{v},\mathbf{b}):(\forall n\in[N])(|v_n|\leq \delta) \wedge(\forall p\in[P-1])(|b_p|\leq \delta)\} \subset [-\pi,\pi]^{N+P-1}.\label{eq:good-region}
\end{equation}
With $Z$ the Gaussian normalizing constant computed earlier, define the measure 
\begin{equation*}
\mu(\mathscr{C})=\frac{1}{Z}\int_\mathscr{C}\exp(-Q(\mathbf{v},\mathbf{b}))d\bsv d\bsb,
\end{equation*}
where $\mathscr{C}$ is any measurable subset of $\mathbb{R}^{N+P-1}$. This is just the multivariate Gaussian law, with mean vector $\mathbf{0}\in\mathbb{R}^{N+P-1}$ and precision matrix $\mathbf{B}/2$. In particular, for any $\mathscr{C}$ we have 
\begin{equation*}
\int_{\mathscr{C}} 4^{NP}\exp(-Q(\mathbf{v},\mathbf{b}))W(\mathbf{v},\mathbf{b})d\bsv d\bsb = 4^{NP}Z\mathbb{E}_\mu\left[WI_{\mathscr{C}}\right],
\end{equation*}
where $\mathbb{E}_\mu[\cdot]$ denotes expectation with respect to probability measure $\mu$ and $I_\mathscr{C}$ is the indicator variable for a point in $\mathbb{R}^{N+P-1}$ lying in $\mathscr{C}$. On $\mathscr{R}^c:=[-\pi,\pi]^{N+P-1}\setminus \mathscr{R}$, we have $\int_{\mathscr{R}^c}4^{NP}\exp(-Q(\mathbf{v},\mathbf{b}))W(\mathbf{v},\mathbf{b})d\bsv d\bsb=4^{NP} Z\mathbb{E}_\mu[WI_{\mathscr{R}^c}]\leq 4^{NP} Z\cdot \mu(\mathscr{R}^c)$, where we used the fact that $W\leq 1$. We claim that there exists $c>0$ such that $\mu(\mathscr{R}^c)\leq 2(N+P)\exp(-cm^{2\varepsilon})\asymp O(m)\exp(-cm^{2\varepsilon})$, which implies that the contribution of the integral is at most $4^{NP}Z\cdot O(m)\exp(-cm^{2\varepsilon})$. To prove our claim, observe that under $\mu$, each coordinate $v_n$ and $b_p$ is univariate Gaussian, with variances determined by the diagonal entries of the matrix $\mathbf{\Sigma}=2\mathbf{B}^{-1}$: $\Var_\mu(v_n)=\Sigma_{n,n}$ and $\Var_\mu(b_p)=\Sigma_{(N+p),(N+p)}$. By Lemma \ref{lemma:covariance-matrix-properties}, these are $\Sigma_{n,n}=2(N+P-1)/NP$ and $\Sigma_{(N+p),(N+p)}=4/N$. The union bound and Chernoff bound together imply that
\begin{eqnarray*}
\mu(\mathscr{R}^c) & = & \mathbb{P}\left(\bigcup_{n,p}\{\pi\geq |v_n|> \delta\}\cup\{\pi\geq |b_p|> \delta\}\right)\\
    &\leq & \mathbb{P}\left(\bigcup_{n,p}\{|v_n|> \delta\}\cup\{|b_p|> \delta\}\right) \\
    & \leq & \sum_{n=1}^N\mathbb{P}(|v_n|> \delta) + \sum_{p=1}^{P-1}\mathbb{P}(|b_p|>\delta) \hspace{6cm}(\text{union bound})\\
    & \leq & 2N\exp\left(-\frac{\delta^2}{2\cdot [2(N+P-1)/NP]}\right) + 2(P-1)\exp\left(-\frac{\delta^2}{2\cdot (4/N)}\right) ~~(\text{Chernoff bound})\\
    & = & 2N\exp\left(-\frac{NPm^{-1+2\varepsilon}}{4(N+P-1)}\right)+2(P-1)\exp\left(-\frac{Nm^{-1+2\varepsilon}}{8}\right).
\end{eqnarray*}
Because $N=\Theta(P)$, the terms $NP/[4(N+P-1)]$ and $N/8$ are of order $m$, hence there is some $c>0$ for which both terms are at least $cm$. For this choice of $c$ we get the upper bound $(2N+2P-2)\exp(-cm^{2\varepsilon})< 2(N+P)\exp(-cm^{2\varepsilon})$, as desired.

We now turn our attention to $\mathscr{R}$. To obtain a tight estimate of the integral $\int_{\mathscr{R}} 4^{NP}\exp(-Q)W d\bsv d\bsb$, we rely once again on Taylor's theorem. Let $c_0\in(0,\pi)$. Because the Taylor series of $\ln(h(t))$ converges uniformly on any compact interval contained within $(-\pi,\pi)$, there exists $C>0$ such that for all $|t|\leq c_0$, 
\begin{equation*}
\ln(h(t))\geq \ln4 -\frac{t^2}{4}-\frac{t^4}{96}-Ct^6.
\end{equation*}
On $\mathscr{R}$, each $t_{np}$ satisfies $|t_{np}|\leq |v_n|+|b_p|\leq 2\delta$. Thus, for large enough $m$, we have $c_0\geq 2m^{-1/2+\varepsilon}=2\delta$, and so 
\begin{equation*}
\sum_{n=1}^N\sum_{p=1}^P\ln(h(t_{np})) \geq NP\ln4 - Q(\mathbf{v},\mathbf{b})-Q_4(\mathbf{v},\mathbf{b}) - E_6,
\end{equation*}
where 
\begin{equation*}
Q_4(\mathbf{v},\mathbf{b})=\frac{1}{96}\sum_{n=1}^N\sum_{p=1}^P t_{np}^4;\hspace{1cm} E_6=C\sum_{n=1}^N\sum_{p=1}^Pt_{np}^6.
\end{equation*}
Our choice of $\varepsilon=1/12$ guarantees that the contribution by $E_6$ is negligible:
\begin{equation*}
0 \leq E_6\leq C\cdot NP(2\delta)^6 \asymp m^2\cdot m^{-3+6\varepsilon}=m^{-1+6\varepsilon}=\frac{1}{\sqrt{m}}.
\end{equation*}
Thus 
\begin{align*}
\int_{\mathscr{R}} 4^{NP}\exp(-Q)W \, d\bsv \, d\bsb
&= \int_{\mathscr{R}} 4^{NP}\exp(-Q)\exp(-Q_4)\exp(-E_6)\, d\bsv \, d\bsb \\
&\in \Big[
4^{NP}e^{-\frac{1}{\sqrt{m}}}
\int_{\mathscr{R}}\exp(-Q)\exp(-Q_4)\, d\bsv \, d\bsb,\;
4^{NP}\int_{\mathscr{R}}\exp(-Q)\exp(-Q_4)\, d\bsv \, d\bsb
\Big].
\end{align*}
% \begin{align*}
% \int_{\mathscr{R}} 4^{NP}\exp(-Q)W d\bsv d\bsb & = &  \int_{\mathscr{R}} 4^{NP}\exp(-Q)\exp(-Q_4)\exp(-E_6) d\bsv d\bsb \\
% & \in & \left[4^{NP}e^{-\frac{1}{\sqrt{m}}}\int_\mathscr{R}\exp(-Q)\exp(-Q_4)d\bsv d\bsb,~~ 4^{NP}\int_\mathscr{R}\exp(-Q)\exp(-Q_4)d\bsv d\bsb\right].& 
% \end{align*}
Writing $\int_\mathscr{R}\exp(-Q)\exp(-Q_4)d\bsv d\bsb = Z\mathbb{E}_\mu[\exp(-Q_4)I_\mathscr{R}]$, we see that 
\begin{align*}
Z\mathbb{E}_\mu[\exp(-Q_4)]~\geq~~~~Z\mathbb{E}_\mu[\exp(-Q_4)I_\mathscr{R}]  &~= & Z\left(\mathbb{E}_\mu[\exp(-Q_4)] - \mathbb{E}_\mu[\exp(-Q_4)I_{\mathbb{R}^{N+P-1}\setminus \mathscr{R}}]\right) \\
    &~\geq & Z\left(\mathbb{E}_\mu[\exp(-Q_4)] - \mu(\mathbb{R}^{N+P-1}\setminus \mathscr{R})\right) \\
    &~\geq & Z(\mathbb{E}_\mu[\exp(-Q_4)]-O(m)\exp(-cm^{2\varepsilon})),
\end{align*}
where the left inequality follows from $e^{-Q_4}>0$, the first right inequality follows from $e^{-Q_4}\leq 1$, and the second right inequality follows from the Chernoff bound used earlier. Because $O(m)\exp(-cm^{2\varepsilon})\rightarrow 0$ as $m\to\infty$, this shows that $Z\mathbb{E}_\mu[\exp(-Q_4)I_\mathscr{R}]$ is approximately $Z\mathbb{E}_\mu[\exp(-Q_4)]$, so it suffices to approximate the latter quantity. 

Write $Y=Q_4-\mathbb{E}_\mu[Q_4]$, so that $Z\mathbb{E}_\mu[\exp(-Q_4)] = Z\exp(-\mathbb{E}_\mu[Q_4])\cdot\mathbb{E}_\mu[\exp(-Y)]$. We claim that the following statements are true.
\begin{enumerate}
\item $\mathbb{E}_\mu[Q_4]=(N+P-1)^2/(8NP)$.
\item $\mathbb{E}_\mu[\exp(-Y)]=1+O(1/m)$. 
\end{enumerate}
For Statement 1, owing to $\mathbb{E}_\mu[Q_4]=1/96\times \sum_{n=1}^N\sum_{p=1}^P \mathbb{E}_\mu[t_{np}^4]$, we must compute $\mathbb{E}_\mu[t_{np}^4]$. We do so in two cases: $p=P$ and $p\leq P-1$. When $p=P$, $t_{np}=v_n$. Recall that a Gaussian random variable $X\sim N(0,\sigma^2)$ satisfies $\mathbb{E}[X^4]=3\sigma^4$, so by Lemma \ref{lemma:covariance-matrix-properties}, $\mathbb{E}_\mu[t_{np}^4] = 12(N+P-1)^2/(NP)^2$. When $p\leq P-1$, $t_{np}=v_n+b_p$. Both $v_n$ and $b_p$ are univariate mean-zero Gaussian, so their sum is univariate mean-zero Gaussian with variance $\text{Var}(v_n+b_p)=\text{Var}(v_n)+\text{Var}(b_p)+2\text{Cov}(v_n,b_p)=2(N+P-1)/(NP)$, where the covariance and variance quantities are derived in Lemma \ref{lemma:covariance-matrix-properties}. So $v_n+b_p\sim N(0,2(N+P-1)/(NP))$ and $\mathbb{E}_\mu[t_{np}^4]=12(N+P-1)^2/(NP)^2$ in this case. Plugging these expressions back into $\mathbb{E}_\mu[Q_4]$, we obtain 
\begin{equation*}
\mathbb{E}_\mu[Q_4]=  \frac{1}{96}\cdot NP \cdot \frac{12(N+P-1)^2}{(NP)^2} = \frac{(N+P-1)^2}{8NP},
\end{equation*}
as desired.

For Statement 2, first observe that $\mathbb{E}_\mu[\exp(-Y)]\geq 1$, because $\exp(-y)\geq 1-y$ for any real $y$ and we can take expectations over $Y$. We thus focus on obtaining an upper bound on $\mathbb{E}_\mu[\exp(-Y)]$ of the form $1+\text{``extra terms''}$, and show that the extra terms decay like $O(1/m)$. We apply Taylor's theorem yet again, this time with integral remainder. For any $y\in\mathbb{R}$, 
\begin{equation*}
e^{-y}=1-y+\frac{y^2}{2} + R_3(y),\hspace{1cm} R_3(y) = \frac{y^3}{2}\int_{0}^1(1-s)^2e^{-sy}ds.
\end{equation*}
We bound $R_3(y)$ above as follows. Observe that if $y\geq 0$, then $\exp(-sy)\leq 1$ for all $s\in[0,1]$. On the other hand, if $y<0$ then $-sy\leq -y$ so $\exp(-sy)\leq \exp(-y)$. Hence, for any $y\in\mathbb{R}$ and $s\in[0,1]$, $\exp(-sy)\leq 1+\exp(-y)$. Consequently,
\begin{equation*}
|R_3(y)| \leq \frac{|y|^3}{2}\int_{0}^1(1-s)^2|e^{-sy}|ds \leq \frac{|y|^3}{2}(1+e^{-y})\int_{0}^1(1-s)^2ds =\frac{|y|^3(1+e^{-y})}{6},
\end{equation*}
where the last equality follows from $\int_0^1(1-s)^2ds=1/3$. Taking expectation over $Y$ and applying $\mathbb{E}_\mu[Y]=0$, we obtain
\begin{equation}
\mathbb{E}_\mu[e^{-Y}] \leq 1+\frac{\mathbb{E}_\mu[Y^2]}{2} + \frac{\mathbb{E}_\mu[|Y|^3(1+e^{-Y})]}{6}.\label{eq:taylor-upper-exp-Y}
\end{equation}
We next estimate each quantity involving an expectation on the right hand side of Eq.~\eqref{eq:taylor-upper-exp-Y}, where we must show that their overall contribution is $O(1/m)$. 

We first estimate $\mathbb{E}_\mu[Y^2]$, which is $\text{Var}_\mu(Q_4)$ because $Y$ is the mean-centered version of $Q_4$. Since $Q_4$ is a sum of $t_{np}^4$'s, $\text{Var}_\mu(Q_4)$ is, up to a scalar multiple, equal to the sum $\sum_{n_1,n_2,p_1,p_2} \text{Cov}_\mu(t_{n_1p_1}^4,t_{n_2p_2}^4)$. Recall that each $t_{np}$ has variance $2(N+P-1)/(NP)$. By Lemma \ref{lemma:covariance-gaussian-4th}, for jointly Gaussian random variables $(T_1,T_2)$ with shared variance $\sigma^2$ and correlation $\rho$, $\text{Cov}(T_1^4,T_2^4)=\sigma^8(72\rho^2+24\rho^4)$. By applying Lemma \ref{lemma:covariance-matrix-properties} (Properties 2 and 3), we obtain the following values of $\rho$ for four cases.
\begin{itemize}
\item If $(n_1,p_1)$ and $(n_2,p_2)$ share neither row nor column index, then $\text{Corr}(t_{n_1p_1},t_{n_2p_2})=-1/(N+P-1)=O(1/m)$.  
\item If $(n,p_1)$ and $(n,p_2)$ share a row index $n$, then $\text{Corr}(t_{np_1},t_{np_2})=(2N-1)/[2(N+P-1)]=O(1)$.
\item If $(n_1,p)$ and $(n_2,p)$ share a column index $p$, then $\text{Corr}(t_{n_1p},t_{n_2p})=(P-1)/(N+P-1)=O(1)$.
\item In the diagonal case, $\text{Corr}(t_{np},t_{np})=1$.
\end{itemize}
Now count all pairs of $(n,p)$ belonging to each scenario and estimate their contribution to the total sum, using $\text{Cov}(T_1^4,T_2^4)=\sigma^8(72\rho^2+24\rho^4)$ with $\rho$ just worked out and $\sigma^2=2(N+P-1)/(NP)=O(1/m)$. 
\begin{itemize}
\item There are $NP$ diagonal terms, each $\text{Var}(t_{np}^4)=O(\sigma^8)=O(m^{-4})$. This contributes $O(m^{-2})$.
\item There are $N{P\choose 2}\asymp m^3$ row-sharing pairs, each covariance $\Theta(\sigma^8)=\Theta(m^{-4})$. This contributes $\Theta(m^{-1})$.
\item There are $P{N\choose 2}\asymp m^3$ row-sharing pairs, each covariance $\Theta(\sigma^8)=\Theta(m^{-4})$. This contributes $\Theta(m^{-1})$.
\item There are $\Theta(m^4)$ pairs that share neither row nor column index, each covariance $O(\sigma^8\rho^2)=O(m^{-6})$. This contributes $O(m^{-2})$.
\end{itemize}
Thus the total sum is 
\begin{equation*}
\mathbb{E}_\mu[Y^2]=\text{Var}_\mu(Q_4)=\sum_{n_1,n_2,p_1,p_2}\text{Cov}_\mu(t_{n_1p_1}^4,t_{n_2p_2}^4) = \Theta\left(\frac{1}{m}\right).
\end{equation*}
We next estimate $\mathbb{E}_\mu[|Y|^3(1+\exp(-Y))]/6$. Because $Q_4$ is non-negative, $Y\geq -\mathbb{E}_\mu[Q_4]=-(N+P-1)^2/(8NP)$. So $\exp(-Y)\leq \exp((N+P-1)^2/(8NP))=\Theta(1)$, i.e., there exists a constant $C_2>0$ such that $\mathbb{E}_\mu[|Y|^3(1+\exp(-Y))]/6\leq C_2\mathbb{E}_\mu[|Y|^3]/6$. We shall show that $\mathbb{E}_\mu[|Y|^3]= O(1/m)$. By H\"{o}lder's inequality with exponents $p=4/3$ and $q=4$ (so $p^{-1}+q^{-1}=1$),
\begin{equation}
\mathbb{E}_\mu[|Y|^3] = \mathbb{E}_\mu[|Y|^3\cdot 1] \leq \left(\mathbb{E}_\mu[|Y|^{3p}]\right)^{\frac{1}{p}}\left(\mathbb{E}_\mu[1^{q}]\right)^{\frac{1}{q}}=\left(\mathbb{E}_\mu[Y^4]\right)^{\frac{3}{4}}.\label{eq:holder}
\end{equation}
Now, we apply a Gaussian hypercontractivity argument (see, e.g., \cite[Chapter 5 \& Theorem 5.10]{janson1997gaussian}): because $\mu$ is a centered non-degenerate Gaussian law on $\mathbb{R}^{N+P-1}$, there exists an invertible matrix $M$ and a standard Gaussian vector $G\sim N(\mathbf{0},\mathbf{I}_{N+P-1})$ such that $(\mathbf{v},\mathbf{b})\overset{d}{=} MG$. Each $t_{np}$ is a linear combination of entries of $(\mathbf{v},\mathbf{b})$, so by composing these linear maps we observe that $t_{np}^4$ is a degree-$4$ polynomial in $G$. This shows that $Y$ belongs to the finite sum of Wiener chaoses up to order $4$. By the hypercontractivity inequality for finite sums of chaoses (\cite[Proposition 2.6]{nourdin2010invariance}), 
\begin{equation}
(\mathbb{E}_\mu[Y^4])^{\frac{1}{4}}=||Y||_{L^4(\mu)}\leq (4-1)^{\frac{4}{2}}||Y||_{L^2(\mu)} = 9(\mathbb{E}_\mu[Y^2])^{\frac{1}{2}}.\label{eq:hypercontractivity}
\end{equation}
Combining Eqs.~\eqref{eq:holder} and \eqref{eq:hypercontractivity}, we get
\begin{equation*}
\mathbb{E}_\mu[|Y|^3] \leq \left(9(\mathbb{E}_\mu[Y^2])^{\frac{1}{2}}\right)^3=729(\mathbb{E}_\mu[Y^2])^{\frac{3}{2}}=\Theta\left(\frac{1}{m^{3/2}}\right),%\label{eq:combined_holder_hypercontract}
\end{equation*}
which is clearly $O(1/m)$. Thus, plugging these estimates back into Eq.~\eqref{eq:taylor-upper-exp-Y} we obtain $\mathbb{E}_\mu[\exp(-Y)]=1+O(1/m)$, which establishes Statement 2.

Putting the pieces together, we have shown that 
\begin{eqnarray*}
|\mathscr{A}_{12}| & = & \frac{1}{(2\pi)^{N+P-1}}\int_{[-\pi,\pi]^{N+P-1}}4^{NP}\exp(-Q)Wd\bsv d\bsb \\
                   & = & \frac{1}{(2\pi)^{N+P-1}}\int_{\mathscr{R}^c}4^{NP}\exp(-Q)Wd\bsv d\bsb + \frac{1}{(2\pi)^{N+P-1}}\int_{\mathscr{R}}4^{NP}\exp(-Q)Wd\bsv d\bsb\\
                   %& = & \frac{4^{NP}Z}{(2\pi)^{N+P-1}}\cdot O(m)\exp(-cm^{2\varepsilon}) + \frac{4^{NP}Z}{(2\pi)^{N+P-1}}e^{-\frac{1}{\sqrt{m}}} e^{-\frac{(N+P-1)^2}{8NP}}\left(1+O\left(\frac{1}{m}\right)\right) \\
                   & = & \frac{4^{NP}N^{(P-1)/2}P^{(N-1)/2}}{\pi^{(N+P-1)/2}}\left[O(m)\exp(-cm^{2\varepsilon}) + e^{-\frac{(N+P-1)^2}{8NP}}\left(1+O\left(\frac{1}{\sqrt{m}}\right)\right)\right]. 
\end{eqnarray*}
Taking base-$2$ logs on both sides, we obtain Eq.~\eqref{eq:a12_approx_semiregular}.
\end{proof}

\noindent\textbf{Remarks on Theorem \ref{thm:growth-rate-a12}.} Another way to set up a generating function in Step 1 of the proof is to obtain a \emph{weighted} count of doubly constrained $\{0,1,2\}$-valued matrices. Concretely, $\mathbf{X}$ is fixed by the two constraints, so one only has to count all valid $\mathbf{A}$'s. Overlaying the ``left'' and ``right'' half matrices --- both $N\times P$ binary matrices --- of $\mathbf{A}$ and taking the sum of entries, the constraints translate into imposing row sums $P$ and column sums $N$ on a $N\times P$ matrix with entries in $\{0,1,2\}$, with any such matrix $\mathbf{A}$ contributing a weight $2^{\#(\mathbf{A},1)}$ to the total count, where $\#(\mathbf{A},1)$ is the number of ones appearing in $\mathbf{A}$. The generating function for this is $\prod_{p=1}^P\prod_{n=1}^N\left(1+2x_ny_p+x_n^2y_p^2\right)=\prod_{p=1}^P\prod_{n=1}^N\left(1+x_ny_p\right)^2$ and the coefficient to extract is $x_1^P\cdots x_N^Py_1^N\cdots y_P^N$. A similar change of variables leads to an integral like RHS of Eq.~\eqref{eq:reduced_integral}. While completing this work, we were informed about \cite{isaev2018complex} that provides an insightful viewpoint: our steps provide an instantiation of the complex martingale framework, where after removing the null directions of the saddle, the Cauchy integral becomes a Gaussian integral with perturbation $f=-Q_4$. Our intermediate step of constructing invariant changes of measure achieves the same goal as their Lemma 4.6: to pull out the contribution of the nullspace of a rank-deficient Gaussian precision matrix. Our asymptotically constant correction term $-\mathbb{E}_\mu[Q_4]$ corresponds to the quantity $\mathbb{E}[f(X)]$ in their Theorem 3.3, while the remaining variance and truncation effects are $o(1)$ under their martingale bounds. One notable difference is that their general approach relies on verifying the existence of a whitening matrix $T$ --- see their Theorems 4.3 and 4.4 --- satisfying technical conditions involving $||T||_\infty$ and $||T||_1$ (which in our setting amounts to controlling the moments of the multivariate truncated Gaussian), whereas we use a hypercontractivity argument to provide explicit control of the higher-order remainder terms, allowing quantification of the $o(1)$ contribution in terms of $N$ and $P$.   

\subsection{Different Independence Heuristic for Admixed Arrays}
\label{subsec:probabilistic_interpretation}

Theorem \ref{thm:growth-rate-a12} admits a probabilistic interpretation. Suppose one samples an admixed array uniformly at random from a set of admixed arrays. The first question is to decide which set to sample from. One might be tempted to use $\mathscr{A}_0$, which has cardinality $2^{4NP}$ (there are $4NP$ cells across $\mathbf{A}$ and $\mathbf{X}$ that can take on values $0$ or $1$). A more appropriate choice is to use an analogue of the normalization factor underpinning the independence heuristic of semi-regular constrained binary and integer matrix enumeration (see Proposition \ref{prop:indep-heur}). This set corresponds to the \emph{smallest} {\red ``common ambient space'' of admixed arrays that properly includes the constraints characterizing $\mathscr{A}_1$ and $\mathscr{A}_2$}. For our semi-regular $1/2$ case, this is the set of admixed arrays $[\mathbf{A},\mathbf{X}]$ for which $\mathbf{A}$ contains $NP$ ones and there are no constraints on $\mathbf{X}$. Thus, we set the normalization factor to the size of this minimal unconstrained superset, $D=2^{2NP}{2NP\choose NP}$.

Sampling uniformly from the minimal unconstrained superset, the probabilities of observing an array with row local ancestry tallies $A_{n\cdot}=P$ (equivalently, $\overline{a_{n\cdot}}=1/2$), an array with ancestry-specific allele dosages $\phi_{p,0}=\phi_{p,1}=N$ (equivalently, $f_{p,0}=f_{p,1}=1/2$), and an array with both these constraints, are respectively $p_1=D^{-1}|\mathscr{A}_1|,p_2=D^{-1}|\mathscr{A}_2|$ and $p_{12}=D^{-1}|\mathscr{A}_{12}|$. Letting the corresponding events be denoted by $\mathcal{E}_1,\mathcal{E}_2$ and $\mathcal{E}_1\wedge \mathcal{E}_2$, one might ask whether $\mathcal{E}_1$ and $\mathcal{E}_2$ are independent, which if true, would imply $p_1p_2=\mathbb{P}(\mathcal{E}_1)\cdot \mathbb{P}(\mathcal{E}_2)=\mathbb{P}(\mathcal{E}_1\wedge \mathcal{E}_2)=p_{12}$. Our theorems imply that, at least in the $N=\Theta(P)$ regime, $\mathcal{E}_1$ and $\mathcal{E}_2$ are \emph{asymptotically negatively correlated}: $|\mathscr{A}_{12}|<D^{-1}|\mathscr{A}_1||\mathscr{A}_2|$. We state this more precisely below.
\begin{Cor}[Independence heuristic overcounts by $e^{1/4}$ factor]
\label{cor:indep-heur-semiregular}
Let $\mathscr{A}_{12}$ be the family of all admixed arrays on $N\times 2P$ binary matrix pairs, such that all row local ancestry tallies are equal to $N$, and all ancestry-specific allele dosages sum to $P$. Equivalently, the normalized marginal constraints $\overline{a}=f_0=f_1=1/2$. Then there is a correction factor $k_{N,P}>0$ satisfying $|\mathscr{A}_{12}|=k_{N,P}D^{-1}|\mathscr{A}_1||\mathscr{A}_2|$, with $\lim_{N,P\to\infty}k_{N,P}=1/\sqrt[4]{e}\approx 0.779$.
\end{Cor}
We directly verify Corollary \ref{cor:indep-heur-semiregular} by comparing $\log(|\mathscr{A}_{12}|)=\alpha_{12}$ and $\log(D^{-1}|\mathscr{A}_{1}||\mathscr{A}_2|)=\alpha_1+\alpha_2-\log D$. Note from Propositions \ref{prop:a1} and \ref{prop:a2} that the semi-regular $1/2$ setting gives $\alpha_1=N \log {2P\choose P} + 2NP$ and $\alpha_2=P\log{2N\choose N}$. Using a detailed Stirling expansion (derived from 6.1.37 in Chapter 6 of \cite{abramowitz1965handbook})
\begin{equation*}
\log{2m\choose m} = 2m - \frac{1}{2}\log(\pi m) - \frac{1}{8m}\log(e) + O(m^{-3}),
\end{equation*}
and applying it with $m=P,m=N$ and $m=NP$, we get 
\begin{equation}
\alpha_1+\alpha_2-\log D = 2NP - \frac{1}{2}\left[N\log(\pi P) + P \log(\pi N) - \log(\pi NP)\right] - \left(\frac{N}{8P}+\frac{P}{8N}-\frac{1}{8NP}\right)\log(e) + o(1).\label{eq:indep-prod}
\end{equation}
Comparing Eq.~\eqref{eq:indep-prod} with Eq.~\eqref{eq:a12_approx_semiregular} in Theorem \ref{thm:growth-rate-a12} for $\alpha_{12}$ and ignoring the terms that are $o(1)$, we obtain their difference,
\begin{equation*}
\left[\frac{(N+P-1)^2}{8NP}-\left(\frac{N}{8P}+\frac{P}{8N}-\frac{1}{8NP}\right)\right]\log(e)=\frac{(N-1)(P-1)}{4NP}\log(e).
\end{equation*}
Therefore
\begin{eqnarray*}
\log(k_{N,P}) &=&\alpha_{12}-(\alpha_1+\alpha_2-\log{D}) \\
              &=& -\frac{(N-1)(P-1)}{4NP}\log(e) + o(1) \overset{N,P\to\infty}{\longrightarrow} -\frac{1}{4}\log(e),
\end{eqnarray*}   
which implies the correction factor after exponentiation. We verify this numerically in the next Section.

\subsection{Algorithms for Exact Computation and Numerical Comparisons}
\label{subsec:algorithms}

To evaluate our approximation (Eq.~\eqref{eq:a12_approx_semiregular}), we implement and extend the \software{EXACT} algorithm \cite{miller2013exact}, which enumerates constrained binary matrices efficiently using dynamic programming and the Gale-Ryser criterion (see Algorithm \ref{alg:geta12-parallel} in Appendix \ref{appsec:algorithm}). We use \software{EXACT} to compute the $|\mathscr{A}|$ terms appearing in Eq.~\eqref{eq:a12_exact} of Proposition \ref{prop:a12}, while summing over the vectors in the permissible set $\mathscr{S}$ in Eq.~\eqref{eq:a12_exact} by parallelizing over the first coordinate of these vectors. A crucial efficiency driver of \software{EXACT} is a \emph{memoization} principle involving the caching of conjugate vectors. Conjugate vectors track the number of times a particular integer appears in a target vector, and are repeatedly used in the dynamic programming recursion. Because any two problems that arrive at the same conjugate vector will have identical counts for the remaining subproblem, avoiding repeated calculations on a previously seen conjugate vector drastically reduces the number of operations. We apply the same memoization strategy independently to each thread in our parallelization, along with each unique conjugate vector within the thread. Table \ref{table:algorithm-comparison} summarizes key features and differences in the \software{EXACT} algorithm for binary matrix enumeration and our parallelized algorithm for admixed array enumeration.        
\begin{table}[ht]
\centering
\caption{Differences and similarities between \software{EXACT} and our algorithm for enumerating doubly constrained admixed arrays.}
\label{table:algorithm-comparison}
\small
\begin{tabular}{@{}lll@{}}
\toprule
& \software{EXACT} \cite{miller2013exact} & Admixed Array Exact Enumeration \\
\midrule
Counts & Binary matrices & Admixed arrays $[\mathbf{A}, \mathbf{X}]$ \\
Constraints & $(\mathbf{r}, \mathbf{c})$ & $(\bsa, \bsphi_0, \bsphi_1)$ \\
Core method & Dynamic programming with memoization & Weighted sum over \software{EXACT} calls \\
Parallelism & None & Over first-coordinate pairs of feasible set vector \\
Memo scope & Global & Per thread, per conjugate vector \\
\bottomrule
\end{tabular}
\end{table}

Using our algorithm, we compute exact values of $|\mathscr{A}_{12}|$ in the semi-regular $1/2$ setting for small values of $(N,P)$, in addition setting $N=P$. By comparing them against Eq.~\eqref{eq:a12_approx_semiregular} (quantities are reported in Table \ref{table:a12_asymptotics}), we observe good quantitative agreement. For example, when $N=P=9$, the saddle-point approximation is about $1.01$ times the exact value (difference in base-$2$ $\log$ scale is $0.015$). Additionally, we calculate the difference between the independence heuristic described in Section \ref{subsec:probabilistic_interpretation} and the saddle-point approximation, the latter converging to $\alpha_{12}$ in the $N=\Theta(P)$ regime. We observe that the difference approaches $-\log(e)/4\approx -0.36067$, the base-$2$ logarithm of the correction factor. For example, when $N=P=10^4$ (a scale comparable to contemporary genetic sequencing data) the difference is $-0.36060$.   
\begin{table}[!ht]
\caption{Comparison of exact and approximate values of $|\mathscr{A}_{12}|$, in the semi-regular $1/2$ case with $N=P$. Values in the saddle-point approximation column are computed using Eq.~\eqref{eq:a12_approx_semiregular}, including all terms except for the $O(1/\sqrt{m})$ quantity. Values in the right-most column are computed by subtracting $(\alpha_1+\alpha_2-\log D)$ from the quantity in the saddle-point approximation column. Exact values of $\alpha_{12}$ for $N=P\in\{10^2,10^3,10^4\}$ were not computed owing to a prohibitively long runtime.}
\label{table:a12_asymptotics}
\vspace{0.2cm}
\centering
\resizebox{\columnwidth}{!}{%
\begin{tabular}{|c|c|c|c|}
\hline
\multicolumn{1}{|c|}{\begin{tabular}[c]{@{}c@{}}Dimension of admixed array\\ ($N=P$)\end{tabular}} & \multicolumn{1}{c|}{$\alpha_{12}=\log|\mathscr{A}_{12}|$} & \multicolumn{1}{c|}{\begin{tabular}[c]{@{}c@{}}Saddle-point approximation \\ (SPA) of $\alpha_{12}$\end{tabular}} & \multicolumn{1}{c|}{\begin{tabular}[c]{@{}c@{}}Difference of SPA \\ with $(\alpha_1+\alpha_2-\log D)$\end{tabular}} \\ \hline
2                                                                                                & 4.16993                       & 4.11700 & $-0.09357$  \\ \hline
3                                                                                                & 10.20945                      & 10.20040  & $-0.16191$  \\ \hline
4                                                                                                & 19.65300                      & 19.66748  & $-0.20380$  \\ \hline
5                                                                                                & 32.67747                      & 32.69626 & $-0.23142$ \\ \hline
6                                                                                                & 49.36714                      & 49.38583   & $-0.25088$ \\ \hline
7                                                                                                & 69.78153                      & 69.79918 & $-0.26529$ \\ \hline
8                                                                                                & 93.96335                      & 93.97978 & $-0.27637$ \\ \hline
9                                                                                                & 121.94419                              & 121.95946 &  $-0.28516$  \\ \hline
$\vdots$                                                                                                & $\vdots$                              & $\vdots$ & $\vdots$                                                                                                  \\ \hline
$10^2$                                                                                                & -                             & $1.91772\times10^4$ &  $-0.35350$ \\ \hline
$10^3$                                                                                                & -                             & $1.98839\times 10^6$ &  $-0.35995$  \\ \hline
$10^4$                                                                                                & -                             & $1.99851\times 10^8$ &  $-0.36060$  \\ \hline
\end{tabular}
}
\end{table}

\section{Discussion}
\label{sec:discussion}

We have introduced admixed arrays, a class of discrete structures arising from analyses of genetic data from admixed populations. For two families of marginal constraints induced by standard genetic summaries, we obtained exact expressions for the number of admissible arrays and derived corresponding asymptotic formulas. These results provide an exact finite-size criterion, and a sharp entropy-based approximate criterion, for determining when row constraints are more restrictive than paired column constraints. We further quantified the error fraction of the entropy approximation in the semi-regular case. It is possible that the $\sqrt{\log(N)/N+\log(P)/P}$ error fraction holds in general.

{\red We have also studied the asymptotic enumeration of admixed arrays subject to both row and column constraints. In the semi-regular $1/2$ setting, we derived a detailed asymptotic expansion of the doubly constrained set using a saddle-point approximation, together with nontrivial control of the non-quadratic remainder terms. Critically, in the dense regime $N=\Theta(P)$ with $\overline{a}=f_0=f_1=1/2$, the row- and column-constraint events satisfy the independence heuristic with a different correction factor than those typically observed for binary and integer matrices under analogous constraints and asymptotic regimes. For general choices of $(\overline{a},f_0,f_1)$ satisfying $\overline{a}=f_0=f_1=\Theta(1)$, I conjecture that the independence heuristic does \emph{not} hold or exist, in part because defining a common ambient space for the normalizing constant $D$ is generally not straightforward. It also remains open how high-dimensional scalings other than $N=\Theta(P)$ and $\overline{a}=f_0=f_1=\Theta(1)$ affect the asymptotic behavior of the doubly constrained set. Investigating analytical and computational strategies for tackling these questions may be of interest to researchers in Analytic Combinatorics of Several Variables \cite{pemantle2024analytic,hackl2023rigorous}.}

Taken together, our work clarifies how admixed arrays fit into the broader picture of constrained combinatorial matrix models discussed in the Introduction. In many classical settings, distinct families of marginal constraints obey an independence heuristic, with deviations from this behaviour driven by non-uniformity and phase transitions. Admixed arrays provide an example of how modest extensions of classical constrained matrix models in an analogously dense and semi-regular setting give rise to different independence heuristics. It remains to be investigated if the findings on two-way admixed arrays will generalize to $\ell$-way admixed arrays, and, more broadly, what the landscape of asymptotic behaviour looks like for this new class of constrained discrete structures.

\section*{Acknowledgements}
\noindent I am grateful to Robin Pemantle for helpful discussions on saddle-point approximation; to Brendan McKay for insightful discussions on the complex martingale method and generating function representations; and to \'Ursula H\'ebert-Johnson and Catharine Lo for comments on an earlier draft of this work. %I am grateful to Catharine Lo and \'Ursula H\'ebert-Johnson for valuable comments on an earlier version of this work.\\
\section*{Conflict of Interest}

\noindent The author declares no conflict of interest.

\section*{Software and Data Availability} 

\noindent Software for counting admixed arrays is available at \url{https://github.com/alanaw1/zagar}.

\printbibliography
\newpage
%%%%%%%%%%%%%%%%%%%%%%%%%%%%%%%%%%%%%%%%%%%%%%%%%%%%%%%%%%%%%%%%%%%%%%%%
%%%%%%%%%%%%%%%%%%%%%%%%%%%%%%% APPENDIX %%%%%%%%%%%%%%%%%%%%%%%%%%%%%%%
%%%%%%%%%%%%%%%%%%%%%%%%%%%%%%%%%%%%%%%%%%%%%%%%%%%%%%%%%%%%%%%%%%%%%%%%
\begin{appendices}
\renewcommand{\thesection}{\Alph{section}}
\titleformat{\section}{\normalfont\Large\bfseries}{Appendix \thesection}{1em}{}

%%%%%%%%%%%%%%%%%%%%%%%%%%%%%%%%%%%%%%%%%%%%%%%%%%%%%
%%%%%%%%%%%% Appendix A: Proofs of Lemmas %%%%%%%%%%%
%%%%%%%%%%%%%%%%%%%%%%%%%%%%%%%%%%%%%%%%%%%%%%%%%%%%%
\section{Proofs of Auxiliary Lemmas}
\label{appsec:aux-lemmas}

\begin{proof}[Proof of Lemma \ref{lemma:theta-function}]
Let $s=1-f_0-f_1$ and write $\log(x)=\ln x/\ln 2$. A direct calculation gives the Hessian
\begin{equation*}
\nabla^2\theta(f_0,f_1)
=\frac{1}{\ln 2}
\begin{bmatrix}
-\left(\frac1{f_0}+\frac1s\right) & -\frac1s\\[4pt]
-\frac1s & -\left(\frac1{f_1}+\frac1s\right)
\end{bmatrix}.
\end{equation*}
This matrix is real and symmetric, so it has two real eigenvalues. Moreover, on the interior $D^\circ=\{f_0>0,f_1>0,s>0\}$, $\text{tr}(\nabla^2\theta(f_0,f_1))=-(1/f_0+1/f_1+2/s) <0$ and $\det(\nabla^2\theta(f_0,f_1))=1/(f_0f_1)+1/(f_0s)+1/(f_1s)>0$. This implies that $\nabla^2\theta(f_0,f_1)$ has two negative eigenvalues and is thus negative definite; this shows $\theta$ is concave. 

Since $\theta$ extends continuously to the compact convex set $\mathscr{D}$ (with the convention $0\ln 0=0$), it achieves a unique global maximum on $\mathscr{D}$. A direction calculation gives the derivative vector
\begin{equation*}
\nabla\theta(f_0,f_1)=(\partial_{f_0}\theta,\partial_{f_1}\theta)=\left(\frac{1}{\ln 2}\ln\frac{s}{f_0}-1,\frac{1}{\ln 2}\ln\frac{s}{f_1}-1\right).
\end{equation*}
Setting each component to zero to obtain the maximizer, we get $s=2f_0=2f_1$, which implies $f_0=f_1=1/4$. Thus 
\begin{equation*}
\max_{(f_0,f_1)\in\mathscr{D}}\theta(f_0,f_1)=\theta\left(\frac{1}{4},\frac{1}{4}\right)=\frac{1}{4}\log4 + \frac{1}{4}\log4 + \frac{1}{2}\log2-\frac{1}{2}=1.
\end{equation*}
\end{proof}

\begin{proof}[Proof of Lemma \ref{lemma:ratio-approx}]
By performing a Taylor expansion of $f(x)=\ln(x)$ around $1/2$, we may simplify the denominator:
\begin{eqnarray*}
\ln\left(\frac{\frac{1}{2}+\sqrt{c\varepsilon}}{\frac{1}{2}-\sqrt{c\varepsilon}}\right) & = & \ln\left(\frac{1}{2}+\sqrt{c\varepsilon}\right) - \ln\left(\frac{1}{2}-\sqrt{c\varepsilon}\right) \\
                            & \approx & \left(\sqrt{c\varepsilon} - \frac{c\varepsilon}{2}+\frac{(c\varepsilon)^{3/2}}{3}\right)-\left(-\sqrt{c\varepsilon} - \frac{c\varepsilon}{2}-\frac{(c\varepsilon)^{3/2}}{3}\right) \\
                            & = & 2\sqrt{c\varepsilon} + \frac{2}{3}(c\varepsilon)^{3/2}.
\end{eqnarray*}
Thus, 
\begin{equation*}
\frac{\varepsilon}{\ln\left(\frac{\frac{1}{2}+\sqrt{c\varepsilon}}{\frac{1}{2}-\sqrt{c\varepsilon}}\right)} \approx \frac{\varepsilon}{2\sqrt{c\varepsilon} + \frac{2}{3}(c\varepsilon)^{3/2}} \approx \frac{1}{2}\sqrt{\varepsilon/c}.
\end{equation*}
\end{proof}

\begin{proof}[Proof of Lemma \ref{lemma:quadratic-upper-bound}]
Consider the function $\psi(x)=\ln(\cos(x))+x^2/2$ defined on $[-\pi/2,\pi/2]$. It satisfies $\psi(0)=0$. Its derivatives are $\psi'(x)=x-\tan(x)$ and $\psi''(x)=-\tan^2(x)$, so at $x=0$ we have $\psi'(0)=0$ and $\psi''(0)=0$, establishing a local maximum at $x=0$. Since $\psi''(x)\leq 0$ on $[-\pi/2,\pi/2]$ with strict equality if and only if $x=0$, we see that $\psi$ is concave on $[-\pi/2,\pi/2]$ and is thus uniquely maximized at $x=0$ and negative at every other value of $x$ --- in other words, $\ln(\cos(x))\leq -x^2/2$ for all $x\in[-\pi/2,\pi/2]$. Setting $x=t/2$ and multiplying the last inequality by $2$ on both sides, we obtain $\ln(h(t)/4)\leq -t^2/4$ as desired. 
\end{proof}

\begin{proof}[Proof of Lemma \ref{lemma:covariance-matrix-properties}]
\underline{Property 1.} We apply Schur's formula:
\begin{eqnarray*}
\det(\mathbf{B}) & = & \det(P\mathbf{I}_N)\cdot\det\left(N\mathbf{I}_{P-1}-\mathbf{1}_{(P-1)\times N}\left(\frac{1}{P}\mathbf{I}_N\right) \mathbf{1}_{N\times (P-1)}\right) \\
                & = & P^N\det\left(N\mathbf{I}_{P-1}-\frac{N}{P}\mathbf{1}_{(P-1)\times (P-1)}\right) \\
                & = & P^N N^{P-1}\det\left(\mathbf{I}_{P-1}-\frac{1}{P}\mathbf{1}_{(P-1)\times (P-1)}\right),
\end{eqnarray*}
so it suffices to show that $\det(\mathbf{I}_{P-1}-\mathbf{1}_{(P-1)\times (P-1)}/P)=1/P$. Since $\mathbf{I}_{P-1}$ and $-\mathbf{1}_{(P-1)\times (P-1)}/P$ commute, they are simultaneously triangularizable, and so $\det(\mathbf{I}_{P-1}-\mathbf{1}_{(P-1)\times (P-1)}/P)=\prod_{p=1}^{P-1}(1+\lambda_p)$, where $\lambda_p$ are the eigenvalues of $-\mathbf{1}_{(P-1)\times (P-1)}/P$. The all-ones matrix $\mathbf{1}_{(P-1)\times (P-1)}$ has eigenvalues $(P-1)$ (multiplicity $1$; eigenvector is $\mathbf{1}_{(P-1)\times 1}$) and $0$ (multiplicity $P-2$; eigenvectors are $\mathbf{e}_p-\mathbf{e}_{P-1}$ for $p\leq P-2$), so $(\lambda_1,\ldots,\lambda_{P-1})=(-(P-1)/P,0,\ldots,0)$, and we obtain $\prod_{p=1}^{P-1}(1+\lambda_p)=1/P$ as desired. 

\underline{Property 2.} Let $\mathbf{S}=N(\mathbf{I}_{P-1}-\mathbf{1}_{(P-1)\times (P-1)}/P)$ be the Schur complement, whose determinant we computed in the proof of Property 1. Applying the block inverse formula, we obtain
\begin{equation}
\mathbf{B}^{-1}=\begin{bmatrix}
(P\mathbf{I}_{N} - \frac{P-1}{N}\mathbf{1}_{N\times N})^{-1} & \mathbf{M}_1 \\
\mathbf{M}_2 & \mathbf{S}^{-1}\label{eq:B-inverse}
\end{bmatrix},
\end{equation}
where $\mathbf{M}_1$ and $\mathbf{M}_2$ are $(P-1)\times N$ and $N\times (P-1)$ matrices irrelevant to our present computation. Recall the Sherman-Woodbury-Morrison formula \cite[p.~65]{golub2013matrix},
\begin{equation*}
(\mathbf{A}+\mathbf{u}\mathbf{v}^T)^{-1}=\mathbf{A}^{-1}-\frac{\mathbf{A}^{-1}\mathbf{u}\mathbf{v}^T\mathbf{A}^{-1}}{1+\mathbf{v}^T\mathbf{A}^{-1}\mathbf{u}}.
\end{equation*}
Setting $\mathbf{A}=\mathbf{I}_{P-1}$, $\mathbf{u}=\mathbf{1}_{(P-1)\times 1}/\sqrt{P}$ and $\mathbf{v}=-\mathbf{u}$, so that $\mathbf{1}_{P-1}-\mathbf{1}_{(P-1)\times(P-1)}/P=\mathbf{A}+\mathbf{u}\mathbf{v}^T$, we obtain the simplification $(\mathbf{1}_{P-1}-\mathbf{1}_{(P-1)\times(P-1)}/P)^{-1}=\mathbf{I}_{P-1}+\mathbf{1}_{(P-1)\times(P-1)}$. Thus $\mathbf{S}^{-1}=(\mathbf{I}_{P-1}+\mathbf{1}_{(P-1)\times(P-1)})/N$, and we obtain $\Sigma_{(N+p),(N+p)}=2(1+1)/N=4/N$. Finally, by setting $\mathbf{A}=P\mathbf{I}_N$, $\mathbf{u}=\sqrt{(P-1)/N}\mathbf{1}_{N\times 1}$ and $\mathbf{v}=-\mathbf{u}$ in the Sherman-Woodbury-Morrison formula and performing similar calculations, we obtain $\Sigma_{n,n}=2(N+P-1)/NP$.

\underline{Property 3.} We work out off-diagonal elements of Eq.~\eqref{eq:B-inverse}. We already computed $\mathbf{S}^{-1}$; this gives $\Sigma_{(N+p_1),(N+p_2)}=2/N$. The Sherman-Woodbury-Morrison formula applied to $(P\mathbf{I}_{N} - \frac{P-1}{N}\mathbf{1}_{N\times N})^{-1}$ yields $(P\mathbf{I}_{N} - \frac{P-1}{N}\mathbf{1}_{N\times N})^{-1}=\mathbf{I}_N/P+(P-1)\mathbf{1}_{N\times N}/(NP)$. This gives $\Sigma_{n_1,n_2}=2(P-1)/(NP)$. By the block inverse formula, the off-diagonal block matrices in Eq.~\eqref{eq:B-inverse} are
\begin{eqnarray*}
\mathbf{M}_1 & = & -\left(P\mathbf{I}_{N} - \frac{P-1}{N}\mathbf{1}_{N\times N}\right)^{-1}\mathbf{1}_{N\times (P-1)} \left(\frac{1}{N}\mathbf{I} _{P-1}\right), \\
\mathbf{M}_2 & = & -\mathbf{S}^{-1}\mathbf{1}_{(P-1)\times N}\left(\frac{1}{P}\mathbf{I}_N\right).
\end{eqnarray*}
Because $\mathbf{B}^{-1}$ is symmetric (it is, up to a scaling factor, a covariance matrix), it suffices to simplify $\mathbf{M}_2$. 
\begin{eqnarray*}
-\mathbf{S}^{-1}\mathbf{1}_{(P-1)\times N}\left(\frac{1}{P}\mathbf{I}_N\right) & = & -\frac{1}{NP}\left(\mathbf{I}_{P-1} + \mathbf{1}_{(P-1)\times(P-1)}\right)\mathbf{1}_{(P-1)\times N} \\
    & = & \frac{-1}{N}\mathbf{1}_{(P-1)\times N}.
\end{eqnarray*}
This gives $\Sigma_{n,N+p}=\Sigma_{N+p,n}=-2/N$.
\end{proof}

\begin{proof}[Proof of Lemma \ref{lemma:covariance-gaussian-4th}]
Recall the Isserlis formula \cite{isserlis1918formula}, which states that for mean-zero jointly Gaussian random variables $(X_1,\ldots,X_{2n})$,
\begin{equation}
\mathbb{E}\left[\prod_{i=1}^{2n} X_i\right] =\sum_{\pi\in\mathscr{M}}\prod_{\{i,j\}\in\pi}\mathbb{E}\left[X_iX_j\right],\label{eq:isserlis}
\end{equation}
where $\mathscr{M}$ is the set of pairings of $\{1,\ldots,2n\}$. Setting $n=4$, $X_1=X_2=X_3=X_4=T_1$ and $X_5=X_6=X_7=X_8=T_2$ in Eq.~\eqref{eq:isserlis}, we see that each pair $\mathbb{E}[X_iX_j]$ contributes one of $\mathbb{E}[T_1^2]=\sigma^2$, $\mathbb{E}[T_2^2]=\sigma^2$ or $\mathbb{E}[T_1T_2]=\rho\sigma^2$. Thus, each summand in Eq.~\eqref{eq:isserlis} contributes $\sigma^8\rho^{\nu(\pi)}$, where $\nu(\pi)$ is the number of cross pairs $(T_1,T_2)$ in the pairing. Since we have $4$ $T_1$'s and $4$ $T_2$'s, the number of cross pairs $\nu$ must be even: $\nu\in\{0,2,4\}$. We proceed to count the number of pairings in each case of $\nu$. If $\nu=0$, the $4$ $T_1$'s are paired among themselves and the $4$ $T_2$'s among themselves. There are $[(4-1)!!]^2=9$ such pairings, leading to a contribution of $9\sigma^8$. If $\nu=4$, each $T_1$ is paired with a unique $T_2$. There are $4!=24$ such pairings, leading to a contribution of $24\rho^4\sigma^8$. Finally, if $\nu=2$, exactly two $T_1$'s are paired with two $T_2$'s. Each pairing of this type must be obtained by picking two $T_1$'s to use in the cross pairs (${4\choose 2}=6$ ways), picking two $T_2$'s to use in the cross pairs (${4\choose 2}=6$ ways), and then matching them ($2$ ways). The remaining two $T_1$'s are paired with each other, as are the remaining $T_2$'s. Thus there are $6\times 6\times 2=72$ such pairings, leading to a contribution of $72\rho^2\sigma^8$. Hence, we obtain $\mathbb{E}[T_1^4T_2^4]=9\sigma^8+24\rho^4\sigma^8+72\rho^2\sigma^8$. To finish the proof, recall that for a mean-zero Gaussian, $\mathbb{E}[T_1^4]=\mathbb{E}[T_2^4]=3\sigma^4$. Therefore $\mathbb{E}[T_1^4]\mathbb{E}[T_2^4]=9\sigma^8$, and taking the difference with $\mathbb{E}[T_1^4T_2^4]$ leads to the desired expression. 
\end{proof}

%%%%%%%%%%%%%%%%%%%%%%%%%%%%%%%%%%%%%%%%%%%%%%
%%%%%%%%%%%% Appendix B: Algorithm %%%%%%%%%%%
%%%%%%%%%%%%%%%%%%%%%%%%%%%%%%%%%%%%%%%%%%%%%%
\section{Admixed Array Exact Enumeration Algorithm}
\label{appsec:algorithm}
\begin{algorithm}[H]
\caption{Parallel Computation of $|\mathscr{A}_{12}(N,P;\bsphi_0,\bsphi_1,\bsa)|$}
\label{alg:geta12-parallel}
\begin{algorithmic}[1]
\Require $N$: number of rows; $P$: number of paired columns; $\bsa$: row sums; $\boldsymbol{\phi}_0, \boldsymbol{\phi}_1 \in \mathbb{N}^P$: column constraints
\Ensure Weighted count $|\mathscr{A}_{12}|$

\Statex \textbf{STEP 1: Preprocessing}
\State Sort $\mathbf{r}=(r_i)\gets \text{sort}(\bsa, \text{descending})$
\State $S \gets \sum_{i=1}^{N} r_i$

\Statex \textbf{STEP 2: Enumerate First-Coordinate Pairs}
\State $\mathcal{F}_0 \gets \emptyset$
\For{$s = \bsphi_1[1]$ \textbf{to} $(2N\mathbf{1} - \bsphi_0)[1]$}
    \For{$a = \max(0, s-N)$ \textbf{to} $\min(N, s)$}
        \State $\mathcal{F}_0 \gets \mathcal{F}_0 \cup \{(a, s-a)\}$
    \EndFor
\EndFor

\Statex \textbf{STEP 3: Parallel Processing}
\ForAll{$(v_1^{(1)}, v_2^{(1)}) \in \mathcal{F}_0$ \textbf{in parallel}}
    \State $\mathcal{C} \gets \textsc{EnumerateSuffix}(v_1^{(1)}, v_2^{(1)}, \boldsymbol{\phi}_0, \boldsymbol{\phi}_1, S)$
    \State Group $\mathcal{C}$ by $L(\mathbf{v}_1, \mathbf{v}_2) = \max_{j}(v_1^{(j)}, v_2^{(j)}) + 2$ into $\{\mathcal{C}_L\}$ \Comment{Add $2$ for conjugate vector representation}
    \ForAll{groups $\mathcal{C}_L$}
        \State Initialize memo table $\mathcal{M}_L \gets \emptyset$
        \ForAll{$(\mathbf{v}_1, \mathbf{v}_2) \in \mathcal{C}_L$}
            \State $\mathbf{c} \gets (\mathbf{v}_1, \mathbf{v}_2)$ \Comment{Column sums: $q_j = v_1^{(j)}$, $q_{P+j} = v_2^{(j)}$}
            \State $|\mathscr{A}| \gets \software{Exact}(\mathbf{r}, \mathbf{c}, \mathcal{M}_L)$ \Comment{See \textcite{miller2013exact}}
            \State $W \gets \prod_{j=1}^{P} \binom{s_j}{\phi_{1,j}} \binom{2N - s_j}{\phi_{0,j}}$ where $s_j = v_1^{(j)} + v_2^{(j)}$
            \State $\text{LocalSum} \gets \text{LocalSum} + W \cdot |\mathscr{A}|$
        \EndFor
    \EndFor
\EndFor
\State \Return $\sum \text{LocalSum}$

\Statex
\Function{EnumerateSuffix}{$v_1^{(1)}, v_2^{(1)}, \boldsymbol{\phi}_0, \boldsymbol{\phi}_1, S$}
    \State \Return all $(\mathbf{v}_1, \mathbf{v}_2)$ extending $(v_1^{(1)}, v_2^{(1)})$ such that:
    \Statex \hspace{\algorithmicindent}\hspace{\algorithmicindent} $\phi_{1,j} \leq v_1^{(j)} + v_2^{(j)} \leq 2N - \phi_{0,j}$ for $j = 2, \ldots, P$
    \Statex \hspace{\algorithmicindent}\hspace{\algorithmicindent} $\max(0, s_j - N) \leq v_1^{(j)} \leq \min(N, s_j)$ for $j = 2, \ldots, P$
    \Statex \hspace{\algorithmicindent}\hspace{\algorithmicindent} $\sum_{j=1}^{P} (v_1^{(j)} + v_2^{(j)}) = S$
\EndFunction

\end{algorithmic}
\end{algorithm}

% \renewcommand{\theHsection}{\Alph{section}} % if using hyperref
% \renewcommand{\thesection}{\Alph{section}}
% \titleformat{\section}{\normalfont\Large\bfseries}{Appendix \thesection}{1em}{}
% \renewcommand{\thesubsection}{\thesection.\arabic{subsection}}
% \renewcommand{\theequation}{S\arabic{equation}}
% \setcounter{equation}{0}

\end{appendices}
\end{document}